\documentclass[11pt]{amsart} 
\usepackage[english]{babel}

\usepackage[utf8]{inputenc}
\usepackage[T1]{fontenc}
\usepackage{lmodern}	

\usepackage{amsmath,amsthm,amscd,amsfonts,amssymb}
\usepackage{epsfig,latexsym,graphicx,psfrag}
\usepackage{xcolor}

\usepackage{enumerate}
\usepackage{array} 
\usepackage{float}
\usepackage{soul} 

\usepackage[a4paper, margin = 3cm]{geometry}

\newtheorem{theo}{Theorem}

\newtheorem{coro}{Corollary}
\newtheorem{prop}{Proposition}[section]
\newtheorem{lemm}[prop]{Lemma}
\newtheorem{propri}[prop]{Properties}
\newtheorem{proper}[prop]{Property}
\newtheorem{sublemm}[prop]{Sublemma}

\newtheorem{claim}[prop]{Claim}
\newtheorem{crit}[prop]{Criterium}
\newtheorem{exem}[prop]{Examples}
\newtheorem{clai}{Claim}
\newtheorem{conse}[prop]{Consequences}

\newtheorem*{rem}{Remark}

\newtheorem{defi}[prop]{Definition }

\newtheorem{definota}[prop]{Definition-Notation }
\newtheorem{rema}[prop]{Remark }

\newcommand{\N}{\mathbb{N}} \newcommand{\Z}{\mathbb{Z}} \newcommand{\Q}{\mathbb{Q}} \newcommand{\R}{\mathbb{R}} 

\def \ds {\displaystyle} \def \sct {\scriptstyle} 

\newcommand{\dis}{\displaystyle}

\newcommand{\pbul}{\scriptstyle \bullet}
\newcommand{\LRA}{\Longleftrightarrow}

\newcommand{\id}{\mbox{\textsf{Id}}}
\newcommand{\bds}{\boldsymbol}

\def\ra{\mathop{\longrightarrow}}
\def\Ra{\mathop{\Longrightarrow}}
\def\La{\mathop{\Longleftarrow}}
\def\Lra{\mathop{\Longleftrightarrow}}

\def\Sgn{\mathop{\dis\lhd}}

\newcommand{\IET}{\textsf{IET}}
\newcommand{\QR}{Q}
\newcommand{\E}{E}
\newcommand{\HA}{H_{A,Q}}
\newcommand{\HAp}{H_{A',Q'}}
\newcommand{\iet}{\textsf{iet}}
\newcommand{\iett}{{\sf{iets}} }
\newcommand{\saf}{\textsf{SAF}}
\newcommand{\im}{\textsf{Im}}
\newcommand{\BP}{\textsf{BP}}
\newcommand{\RIET}{\textsf{IET}_{\Q}}
\newcommand{\DIET}{\Delta\textsf{IET}_{\Q}}

\newcommand{\ab}{\textsf{ab}}
\newcommand{\IN}{\{0,\dots,q-1 \} } 
\newcommand{\D}{\mathbb{D}} 
\newcommand{\cl}{\mbox{\text{dl}}}

\newcommand{\bigsigma}{\dis \sigma}
\def\psd{\mathop{\rtimes }}

\definecolor{purple}{rgb}{0.62,0.12,0.94}

\usepackage{lineno} 

\title{Some Elementary Amenable subgroups of interval exchange transformations}

\author{Nancy Guelman and Isabelle Liousse}

\begin{document}

\address{{\bf  Nancy  GUELMAN}, \rm{IMERL, Facultad de Ingenier\'{\i}a, {Universidad de la Rep\'ublica, C.C. 30, Montevideo, Uruguay.} \textit{nguelman@fing.edu.uy}}.}

\address{{\bf Isabelle LIOUSSE}, \rm{Universi\'e de Lille, CNRS  UMR 8524 - Laboratoire Paul Painlevé, F-59000 Lille, France.  \textit {isabelle.liousse@univ-lille.fr}}.}

\begin{abstract}
In this paper, we study a family of finitely generated elementary amenable $\iet$-groups. These groups are generated by finitely many rationals $\iet$s and rotations. For them, we state criteria for not virtual nilpotency or solvability, and we give  conditions to ensure that they are not virtually solvable. We precise their abelianizations, we determine when they are isomorphic to certain lamplighter groups and we provide non isomorphic cases among them. As  consequences, in the class of infinite finitely generated subgroups of $\iet$s up to isomorphism,  we exhibit infinitely many non virtually solvable and non linear groups, and  infinitely many solvable groups of arbitrary derived length.
\end{abstract}

\maketitle

\tableofcontents


\vfill \eject


\section{Introduction}

\begin{defi} An {\bf interval exchange transformation} ($\iet$) is a bijective map $f: [0,1) \rightarrow [0,1)$ defined by a finite partition of the unit interval into half-open subintervals and a reordering of these intervals by translations. 
 If a partition consists of intervals with rational endpoints, we say that $f$ is a {\bf rational interval exchange transformation}.

We denote by $\boldsymbol{\IET}$ the group consisting of all $\iet$s.
\end{defi}

\begin{rema}\label{RemRZ}
As the maps we deal with are only piecewise continuous, it is equivalent (and sometimes more convenient) to consider an \iet \ as  a bijection of the circle $[0,1)/0\sim 1$ (see e.g. \cite{DFG} or  \cite{Cor}). To avoid ambiguity, we will denote by $\IET(\R/\Z)$ the group of all interval exchange transformations on the circle.
\end{rema}

Since the late seventies, the dynamics and the ergodic properties of a single interval exchange transformation were intensively studied (see e.g. the Viana survey \cite{Vi}). A natural extension is to consider the dynamics in terms of group actions. 

\begin{exem} \ 

$\pbul$ {\rm Circular rotations $R_a \ (a\in [0,1) / {\hskip 0mm}_{0\sim 1})$ belong to $\IET(\R/\Z)$, they identify with \iet s over $2$ intervals and they form a subgroup of $\IET$ that we denote by $\mathbb S^1$.}
\end{exem}

$\pbul$ Let $q$ be a positive integer, the symmetric group $\mathcal S_q$ can be represented as a subgroup of $\IET$, due to the following 
\begin{defi} 
Let $q \in \N_{>1}$ and consider the partition $\dis [0,1)=\bigsqcup_{i=1} ^q I_i, \text{ \ where \  } I_i=\bigl[\frac{i-1}{q}, \frac{i}{q}\bigr).$

To any given $\tau \in \mathcal S_q$, we associate the unique $\iet$, $E_\tau$, that is continuous on $I_i$, $i \in \{1,\cdots, q\}$ and that satisfies $\E_\tau(I_i) =I_{\tau(i)}$. We set $\mathcal G_{\frac{1}{q}}:= \{\E_\tau, \ \tau \in \mathcal S_q \}$, 
it is plain that $\mathcal G_{\frac{1}{q}}$ is a subgroup of $\IET$.  In particular, if $\sigma_q$ is the $q$-cycle $(1,2,...,q) \in \mathcal S_q$ then $E_{\sigma_q} = R_{\frac{1}{q}}$.

In addition, the map $\mathfrak S : \left\{ 
\begin{array}{ll}
\mathcal G_{\frac{1}{q}} \to  \mathcal S_q \cr
\E_\tau  \mapsto \tau=\mathfrak S(\E_\tau)
\end{array} \right.$ is an isomorphism.
\end{defi}

\smallskip

This shows that all finite groups and all finitely generated abelian groups can be represented as subgroups of $\IET$. 

\medskip

The most famous problem concerning $\iet$-groups was raised by Katok: "\textit{Does $\IET$ contain copies of $\mathbb F_2$, the free group of rank $2$~?}\ " Dahmani, Fujiwara and Guirardel established that such subgroups are rare (\cite{DFG1} Theorem 5.2). 

\medskip

More generally, one can ask for a description of possible finitely generated subgroups of $\IET$. According to Novak "\textit{there is no distortion in $\IET$}" (\cite{No} Theorem 1.3), and as a standard consequence, "\textit{any finitely generated nilpotent subgroup of $\IET$ is virtually abelian}".

\medskip

Among many things, Dahmani, Fujiwara and Guirardel proved that "\textit{any finitely presented subgroup of $\IET$ is residually finite}" (\cite{DFG1} Theorem 7.1), "\textit{$\IET$ contains no infinite Kazhdan groups}" (\cite{DFG1} Theorem 6.2), and "\textit{any finitely generated torsion free solvable subgroup of $\IET$ is virtually abelian}" (\cite{DFG} Theorem 3). 

\smallskip

In contrast, they provide plenty of non virtually abelian solvable subgroups of $\IET$, by studying and using the embeddings of Lamplighter groups in $\IET$ (\cite[Section 4]{DFG}). 

\medskip

This suggests that finding big subgroups of $\IET$ seems difficult, especially as Cornulier conjectured in \cite{CorBour} that $\IET$ could be amenable. 

\smallskip

This conjecture is motivated by
works in \cite{JM1} and \cite{JMBMS} that establish that: 
"\textit{Given $\alpha_1,\alpha_2 \in  \R/{}_{\dis\Z} \setminus \Q/{}_{\dis \Z}$, the Nekrashevich group $\bigl[\ \Gamma_{\alpha_1,\alpha_2}\  ,\ \Gamma_{\alpha_1,\alpha_2} \ \bigr]$  is a finitely generated amenable simple group}",  where  $\Gamma_{\alpha_1,\alpha_2}$ is the $\IET(\R/\Z)$-subgroup of all transformations whose discontinuity points belong to $\langle\alpha_1,\alpha_2 \rangle$.
Note that, by its simplicity, the Nekrashevich group is not elementary amenable (see \cite[Corollary 2.4]{Chou}).

\smallskip
To summarize, beyond Katok's conjecture, there is no known example of non amenable or torsion-free non virtually abelian finitely generated $\iet$-group.

\medskip

The original motivation of this paper was to construct families of non virtually solvable finitely generated subgroups of $\IET$, we noticed in \cite[Corollary 3(1)]{GLtg} that certain $\iet$-groups are not virtually nilpotent. More precisely, we proved the following
\begin{crit} \label{CritGLtg}
The group generated by an irrational rotation and an $\iet$ that is not a rotation is not virtually nilpotent even not virtually polycyclic.
\end{crit} 

In addition, in \cite[Subsection 7.4]{GLtg}, we construct a finitely generated, elementary amenable and non virtually solvable $\iet$-group. A part of the material developed in that work is taken up, generalised and improved here, in the sense that we provide infinitely many non isomorphic finitely generated and non virtually solvable $\iet$-groups. More precisely, these groups are finitely generated by rotations and rationals $\iet$s. In particular, we prove that they are elementary amenable and we give criteria for their solvability or their non virtual solvability. We precise their abelianizations and we state some results of non isomorphicity.

\medskip

More precisely, the groups in consideration are defined as follows

\begin{defi}\label{defH} Let $s,m\in \N\setminus\{0\}$.
Let $\alpha_1, \cdots, \alpha_s\in \R/{}_{\dis\Z} \setminus \Q/{}_{\dis \Z}$ be $\Q$-independent.

We denote by $A=\langle\alpha_i\rangle$ the additive group generated by the $\alpha_i$'s and by $R_A$ the subgroup of $\IET$ of all rotations $R_a$ with  $a\in A$. The $\Q$-independence of the $\alpha_i$'s implies that $R_A\simeq A \simeq \Z^s$. 

\smallskip

Let $q\in \N\setminus\{0\}$  
and $g_1,...,g_m \in \mathcal G_{\frac{1}{q}}$. We denote by $\QR$ the subgroup of $\mathcal G_{\frac{1}{q}}$ generated by the $g_i$'s.

\smallskip

We define a finitely generated subgroup of $\IET$ by $\boldsymbol{H_{A,\QR}=\langle  R_{\alpha_1}, \dots, R_{\alpha_s} \ , \ g_1,\dots,g_m \rangle}$.

\end{defi}

Quite similar families of $\iet$-groups have been considered in \cite[Section 4]{DFG} within the framework of finitely generated groups, and, for larger $\iet$-groups, by Boshernitzan (\cite{Bos}) and Bier-Sushchanskyy (\cite{Bier}).

\bigskip

Here, we prove the following

\begin{theo}\label{THGenF1}Let $A$, $\QR$ be as in Definition \ref{defH} and $\sigma=\sigma_q$ be the $q$-cycle $(1,2,...,q) \in \mathcal S_q$. Then either:
\begin{enumerate} 
\item $\QR\subset \mathbb S^1$ is a rotation group. This condition is equivalent to that of $\HA$ being abelian, and also it is equivalent to require that $\langle  \mathfrak S(\QR),\sigma \rangle$ is abelian; or
\item $\QR$ contains an $\iet$ that is not a rotation, then
the group $\HA$ is elementary amenable of Chou class $EG_2$, it is not virtually nilpotent, it contains a free semigroup on two generators and it has exponential growth. 
\end{enumerate}
\end{theo}

\medskip

Concerning the question of deciding whether these groups are solvable or are not virtually solvable, we have

\begin{theo}\label{THF0} Let $A$, $\QR$ be as in Definition \ref{defH} and $\sigma$ be the $q$-cycle $(1,2,...,q) \in \mathcal S_q$. Then:

\begin{enumerate}
\item If $q \geq 5$ and $\langle \mathfrak S (\QR),\sigma\rangle$ contains the alternating group $\mathcal A_q$ then the group $\HA$ satisfies Item (2) of Theorem \ref{THGenF1}. Moreover $\HA$ is not virtually solvable and it is not a linear group.

\item Let $p\in \N^*$, the group $\HA$ is $p$-solvable if and only if the group $ \langle \mathfrak S (\QR),\sigma\rangle$ is $p$-solvable.
\end{enumerate}
\end{theo}

As consequences of the following theorem, Theorem \ref{THF0} (2) allows the construction of infinitely many non pairwise isomorphic solvable \iet-groups of arbitrary derived length and Theorem \ref{THF0} (1) provides infinitely many non virtually solvable \iet-groups that are not pairwise isomorphic (see Propositions \ref{nSol} and \ref{nonVSol}). 


\begin{theo} \label{THRFin} Let $A$, $\QR$ be as in Definition \ref{defH}, then: 
\begin{enumerate}
\item The torsion elements of $\HA$ form a normal subgroup $T(\HA)$ of $\HA$. The quotient group $\frac{\HA}{T(\HA)}\simeq A$ and $\HA\simeq T(\HA)\rtimes A$.

\item The abelianization $\frac{\HA}{[\HA,\HA]}$ of $\HA$ is isomorphic to a group of the form  $F \times A$, where $F$ is a quotient of $\mathfrak S(\QR)$ by a normal subgroup $N_Q$ satisfying $[\mathfrak S(\QR), \mathfrak S(\QR)]<N_Q<[W,W] \cap \mathfrak S(\QR)$ with $W=\langle \mathfrak S(\QR), \sigma \rangle$. In particular $F$ is a finite abelian group.

\end{enumerate}
\end{theo}

\medskip

\begin{coro} \label{CPAb} Let $A$ and $\QR$ be as in Definition \ref{defH}, $\sigma=\sigma_q$ be the $q$-cycle $(1,2,...,q) \in \mathcal S_q$ and $W= \langle\mathfrak S(Q),\sigma \rangle$. Then:
\begin{enumerate}
\item If $[W,W] \cap \mathfrak S(Q)=[\mathfrak S(Q),\mathfrak S(Q)]$ then $F\simeq \frac{\mathfrak S(Q)}{[\mathfrak S(Q),\mathfrak S(Q)]}:=\mathfrak S(Q)_{\ab}$. 
In particular, if $\sigma\in \mathfrak S(Q)$ then $F\simeq \mathfrak S(Q)_{\ab}$. 

\item If $\mathfrak S(Q)=\langle \tau \rangle$ with $\tau \in \mathcal S_q\setminus \mathcal A_q$ an involution, then $F\simeq \mathfrak S(Q)$.

\item If $\mathfrak S(Q)=[\mathfrak S(Q),\mathfrak S(Q)]$ then $F=\{\id\}$. In particular, if $\mathfrak S(Q)= \mathcal A_q$ with $q\geq 5$ then $F=\{\id\}$.
\end{enumerate}
\end{coro}

\begin{coro} \label{Noniso} Let $A,A'$ and $\QR, \QR'$ be as in Definition \ref{defH}. Then:
\begin{enumerate}
\item If $A\not\simeq A'$ that is if $s\not=s'$, then $\HA$ and $\HAp$ are not isomorphic. 
\item If $A\not= A'$ then $\HA$ and $\HAp$ are not conjugate in $\IET$.
\item If  $\sigma_q \in  \mathfrak S(\QR)$, $\sigma_{q'} \in  \mathfrak S(\QR')$, and $\QR$, $\QR'$ have non isomorphic abelianizations then $\HA$ and $\HAp$ are not isomorphic. 
\end{enumerate}
\end{coro}

\begin{coro} \label{NonisoLamp} Let $A$ and $\QR$ be as in Definition \ref{defH}.

The group $H_{A,Q}$ is  isomorphic to a Lamplighter group $L\wr G$ with $L$ finitely generated and  $G\simeq \Z^k$ if and only if  \ 
$V:=\bigl\langle \ \sigma^p \mathfrak S(Q)\sigma^{-p}, \ p\in \IN \  \bigr\rangle$ is abelian. In this case, $A\simeq \Z^k$ and $Q\simeq L$ is finite and abelian.
\end{coro}

The proofs are based on two key ingredients: The first one is a morphism $\ell : \HA\to \frac{\R}{ \Z}$ that describes the irrational part of the translations of the $\HA$-elements. The second one is the local permutations which are morphisms defined on $\ker \ell$ into $\mathcal S_q$, that encode the way how the rationals points, $\dis \frac{p}{q}$ with $p\in \IN$, are permuted. This text also contains, as an appendix, generalities and relevant non original material on commutators, wreath products, solvable groups and virtually solvable groups. 

\subsection*{Acknowledgments.} We thank Y. Cornulier for fruitful discussions, especially for explanations about its proof of stability by group extensions for the property of being virtually solvable. This proof is resumed in subsection \ref{AppVR} of the Appendix. We thank M. Belliart for helping us with the construction that leads to Proposition \ref{nSol}.
We acknowledge support from the MathAmSud Project GDG 18-MATH-08, the Labex CEMPI (ANR-11-LABX-0007-01), ANR Gromeov, the University of Lille, the I.F.U.M.I Laboratorio de la Plata and Project CSIC 149/348. 

\section{Preliminaries on \iett and subgroups of rational \iett}

\begin{defi} \ 

Let $f\in \IET$ and $0=a_0<a_1< \cdots < a_i< \cdots < a_m=1$ such that $f_{ \vert_{[a_i, a_{i+1})}} \text{ is a translation.}$

\begin{itemize}
\item The translation vector of $f$ is $\delta(f)=(\delta_i)_{i=1, \cdots ,m}$ defined by 
 $f_{ \vert_{[a_{i-1}, a_{i})}}(x)= x+\delta_i$, for $i\in \{1, \cdots ,m\}$.

\item The length vector of $f$ is $\lambda(f)=(\lambda_i)_{i=1, \cdots ,m}$ where $\lambda_i=a_{i} - a_{i-1}$.

\item The permutation associated to $f$ is $\pi(f)=\pi \in \mathcal S_m$ defined by $f( [a_{i-1}, a_{i}) )=J_{\pi(i)}$ where the $J_j$'s are the ordered images of the intervals $\ [a_{i-1}, a_{i})$.

\smallskip

\ \hskip -1.4cm From now on, we assume that for any $i\in \{1,\cdots,m\}$, the point $a_i$ is  discontinuity point of $f$.

\ \hskip -1.4cm After a possible change of the subdivision $\{a_i\}$, there is no loss of generality in this requirement.

\smallskip

\item The break point set of $f$ is $\BP(f)=\{a_i, \ i=0, \cdots ,m-1\}$.

\item Let $x\in [0,1)$, the $f$-orbit of $x$ is $\mathcal O_f(x)=\{f^n(x), \ n\in \Z\}$.

\item More generally, let $G$ be a subgroup of $\IET$ and  $x\in [0,1)$, 

\hskip 4mm $\pbul$ the break point set of $G$ is $\BP(G)=\bigcup_{g\in G} \BP(g)$ and 

\hskip 4mm $\pbul$ the $G$-orbit of $x\in [0,1)$ is $\mathcal O_G(x)=\{g(x), \ g\in G\}$.

\end{itemize}
\end{defi}

\begin{definota} \ 
\begin{itemize}

\item The group of \ rationals $\iet$s \ is \ \ $\boldsymbol{\RIET}=\{f\in \IET, \ \lambda_i(f) \in \Q\}$.

\item The group of $\Delta$-rational $\iet$s is $\DIET=\{f\in \IET, \ \delta_i(f) \in \Q\}$.

\end{itemize}

\end{definota}

\medskip

\begin{rema}\label{LenBPTran} We recall that the break points and the translations of an $\iet$ $f$, with associated permutation $\pi$, are related to its lengths by the following formulas. Let $i\in\{1,\cdots m\}$ then 
$$ (1) \ \ a_i=\sum_{j=1}^{i} \lambda_j(f) \text{ \ \ and \ \ } 
(2) \ \ \delta_i(f) =  \sum_{j:\pi(j)<\pi (i)} \hskip -4mm \lambda_j(f)  \ - \ \sum_{j=1}^{i-1} \lambda_j(f)$$
\end{rema}

\begin{propri}\label{base} \ 
\begin{enumerate}
\item $\dis \RIET \subsetneqq \DIET$ and $\dis \DIET\cap \mathbb S^1 =\RIET\cap \mathbb S^1 = \{R_a, a\in \Q\}$.

\item If $Q$ is a finitely generated subgroup of $\RIET $ then there exists $q\in \N^*$ such that $Q<\mathcal G_{\frac{1}{q}}$.

\item $\DIET$ is locally finite, that is, its finitely generated subgroups are finite. In particular, any $f\in \DIET$ has finite order.
\end{enumerate}
\end{propri}

\begin{proof} \ 
\begin{enumerate}
\item $\dis \RIET \subset \DIET$ by Formula (2) of Remark \ref{LenBPTran}. Moreover, for $g\in \RIET \setminus \mathbb S^1$ and $\alpha\notin \Q$, the map $R_\alpha g R_\alpha ^{-1} \in \DIET \setminus \RIET$. Finally, it is obvious that $\dis  \{R_a, a\in \Q\} \subset \RIET\cap \mathbb S^1 \subset \DIET\cap \mathbb S^1 $ ; in addition, if $f \in  \DIET\cap \mathbb S^1$  then $f=R_a$,  $\delta(f) = \{a\} \mod 1$ and $\delta(f)\subset \Q$, therefore $a \in \Q$.

\medskip

\item Let $Q= \langle g_j, \ j=1, \cdots , m\  \rangle$ and $q$ be the least common multiple of the denominators of the lengths of the $g_j$'s. According to  Remark \ref{LenBPTran}, for $f\in Q$ we have $\BP(f) \cup f(\BP(f) ) \subset \{\frac{i}{q}, i=0, \cdots q-1\}$ and then $f\in \mathcal G_{\frac{1}{q}}$.

\medskip

\item Let $D=\langle h_j \rangle$ be a finitely generated subgroup of $\DIET$. By definition all the $\delta_i(h_j)$ are rational and denoting by $q$ the least common multiple of the denominators of the $\delta_i(h_j)$, we  have for any $x \in [0,1)$
$$\mathcal O_{D} (x) \subset \left\{ x+ \frac{p}{q}, \ p\in \Z\right\} \cap [0,1).$$
Then, $\mathcal O_{D} (x) $ is finite and $\BP(D)\subset \cup_j\mathcal O_{D} (\BP(h_j))$ is also finite. Therefore $D$ is finite since there are only finitely many intervals delimited by the points of $\BP(D)$ and then only finitely many possible permutations of these intervals.
\end{enumerate}
\end{proof}

\section{Basic properties of $\HA$} \label{Family 1}

\subsection{The morphism $\ell$ and its kernel}  \ 

\smallskip

From now on, we assume that $H$ is a subgroup of $\IET$ generated by finitely many rotations and rational $\iet$s. According to Properties \ref{base} (2), there exists $A$ and $Q$ as in Definition \ref{defH} such that $H=\HA$. Namely, 
$$\bds{H=\langle  R_{\alpha_1}, \cdots , R_{\alpha_s} \ , \ g_1, \cdots ,g_m \rangle \quad \text{ where} \quad A=\langle \alpha_i \rangle \simeq \Z^s \text{ and  } \QR=\langle g_j\rangle<\mathcal G_{\frac{1}{q}}}$$  

\medskip

Note that any $f\in H$ can be written as 
$$f= R_{a_n} t_n R_{a_{n-1}}t_{n-1} \cdots R_{a_1} t_1 R_{a_0}\mbox{ where } t_j \in Q \mbox{ and } a_j \in A$$
and it is easy to see that the irrational part in the translations of $f$ is constant (i.e. $f(x) \mod \Q$ does not depend on the point $x$). More precisely, we have

\begin{lemm} \label{L1EM} \ 

Let $f \in H$, there exists a unique $\mathfrak n(f)=(\mathfrak n_1(f), ..., \mathfrak n_s(f)) \in \Z^s$ such that for all $x\in [0,1)$, it holds that $$f(x) = x + \sum_{i=1} ^s \mathfrak n_i(f) \alpha_i + \frac{p_f(x)}{q} \mbox{ , \  where } p_f(x) \in \mathbb Z.$$
\end{lemm}

\medskip

\begin{proof} As for the proof of \cite[Lemma 7.2]{GLtg}, we argue by induction on the length $L_S(f)$ of $f\in H$ as a word in $S=\{R_{\alpha_i} (i=1,...,s)  \ ; \  g_j (i=1,...,m) \}$.
\end{proof}

\bigskip \bigskip

\begin{lemm} \label{L2EM} \ 
\begin{enumerate}
\item The map $\ell : H \to \R/\Z, \ f \mapsto \sum \mathfrak n_i(f) \alpha_i$ is a morphism. For any $a\in A$, $\ell(R_a)=a$  and $\ell(H)=A$.
\item Any $f\in H$ can be written as $f= P_f R_{\ell (f)} = R_{\ell(f)} Q_f $ where  $P_f,Q_f \in \ker \ell$.
\item The $\ker \ell$-orbit of $x\in [0,\frac{1}{q})$ is contained in the finite set $\Bigl\{ x+\frac{p}{q}, \ p \in \{0, \cdots , q-1 \}\Bigr\} $.
\item One has $[H,H] \lhd \ker \ell$. Denoting by $T(H)$ the set of the torsion elements of $H$, we have $\ker \ell=T(H)=H\cap \DIET$. In particular, $T(H)$ is a locally finite normal subgroup of $H$
which contains $[H,H]$.

\end{enumerate}
\end{lemm}

\begin{proof} \  

\begin{enumerate}

\item We obviously have $\ell(\id) = 0$ and $\ell(R_{\alpha_i})=\alpha_i$.

Given $f,h \in H$ and $x\in [0,1)$, the computation of $hf(x)$ leads to   
$$hf(x)= x + \ (\ell(f) +\ell (h))\  + \frac{p_f(x)+p_h(f(x))}{q}. \quad \text{ Therefore}$$ 
$$\ell(hf)=\ell (f) +\ell (h) \mbox{ \quad and \quad} p_{hf}(x)=p_f(x)+p_h(f(x)).$$

In particular, $\ell$ is a morphism from $H$ to $\R/\Z$. This and the first observation imply that    for any $a\in A$, $\ell(R_a)=a$ and $\ell(H)=A$.

\medskip

\item Let $f\in H$, one has $\ds \ell \bigl(f R_{-\ell(f)}\bigr) = \ell (f) +\ell (R_{-\ell(f)})=\ell (f) -\ell(f)=0$. Then \break
{$\dis P_f =f \ R_{-\ell(f)} \in \ker \ell$ \ and $\dis f= P _f \ R_{\ell(f)}$. \hfill\break Similarly, we get $Q_f= R_{-\ell(f)} \ f \in \ker \ell$ \ and $f= R_{\ell(f)} \ Q_f$.}

\item Let $f\in \ker \ell$ and $x\in [0, \frac{1}{q})$, Lemma \ref{L1EM} gives $f(x)= x+\frac{p_f(x)}{q}$ for some $p_f(x)\in \Z$, and $f(x)\in[0,1)$ forces  $p_f(x)\in \{0,1,\cdots, q-1\}$. 

\item As $\ell$ is a morphism into an abelian group and $[H,H]\lhd H$, it is plain that $[H,H]\lhd \ker \ell$.

\noindent In addition, according to the previous item, $\ker \ell$ is contained in $\DIET$ so by Properties \ref{base} (3) it is locally finite. In particular $\ker \ell< T(H)$, and the converse is obvious by contrapositive and Lemma \ref{L1EM}.
\end{enumerate} \vskip -0.6cm \end{proof}

As we are interested in properties of solvability  and virtual solvability, the next Lemma will be useful.

\begin{lemm} \label{CritVSol} \ 
The following assertions are equivalent 
\begin{enumerate}
\item $H$ is virtually solvable.
\item $\ker \ell $ is virtually solvable.
\item $[H,H]$ is virtually solvable.
\end{enumerate} 
\end{lemm}

\begin{proof}
The implications $(1) \Longrightarrow (2)$ and $(2) \Longrightarrow  (3)$ follow from the well-known fact that the class of virtually solvable groups is stable by taking subgroups (see Appendix \ref{AppVR}). 

\medskip

It remains to prove that  $(3) \Longrightarrow (1)$. As $ [H,H]$ is virtually solvable and $H/[H,H]$ is abelian then, according to the stability by extension of the virtually solvable class (see Appendix \ref{AppVR}), we get that $H$ is virtually solvable. \end{proof}

\subsection{Generating sets for $\ker\ell$ and $[H,H]$}
\begin{lemm}  \label{L3EM} \ 
\begin{enumerate}
\item The group $\ker \ell$ is generated by $\bigl\{ \dis R_a g R_a^{-1} \ ; \ g \in \langle g_j \rangle =\QR,  \  a \in A \bigr\}$.

\medskip

\item The group $[H,H]$ is generated by the set
$$\dis  \Bigl\{  \ 
[R_a,t] \ , \ [t_1,t_2] \ \ ; \ \   a \in A \ \text{ and } \  t,t_1,t_2\in \QR \ 
\Bigr\}$$
\end{enumerate}
\end{lemm}

\begin{proof} \ 
\begin{enumerate}
\item On one side, it is clear that any $R_a g R_a^{-1}$, with $g \in \QR$ and $a \in A$, belongs to $\ker \ell$.

On the other side, we previously noted that any $f\in G$ can be written as 
$$f=R_{a_n} t_n R_{a_{n-1}}t_{n-1} \cdots \ R_{a_1} t_1 R_{a_0}\mbox{ where } t_j \in \QR \mbox{ and }  a_j \in A , \mbox{ then }$$
$$f= R_{a_n} t_n R_{a_n} ^{-1}  \ \  R_{a_n+a_{n-1}} t_{n-1}  R_{a_n+a_{n-1}}^{-1} \     \cdots \ R_{a_n+...+a_1} t_1 R_{a_n+\dots+a_1} ^{-1} \ \ R_{a_n+\dots+a_1+a_0} $$

Moreover, as $\ell$ is morphism and  $\ell(R_a g R_a^{-1})=0$, we have $\ell(f)= \ell (R_{a_n+\dots+a_1+a_0} )$. Then  $f\in \ker \ell$ if and only if $\sum a_j=0$, thus any $f\in \ker \ell$ is the product of elements of the form $R_{a} t R_{a} ^{-1}$ with $t \in \QR$ and $a \in A$.
\end{enumerate}

\medskip

Before showing Item (2), we state the following general
\begin{propri} \label{CommPtGene}  
Let $G$ be a group and $a,b, h \in G$, it holds that 
$$[b,a]=[a,b]^{-1}, \ \ h[a,b]h^{-1}=[hah^{-1},hbh^{-1}] \text{ \ and \ } ba=[b,a]ab=ab[b^{-1}, a^{-1}].$$
\end{propri}

\begin{enumerate}
\item[(2)] $-$ We first check that the group 
$\dis H_1=\Bigl\langle \ [R_a,t] \ , \ [t_1,t_2] \ \ ; \ \  \left\{\begin{array}{ll} a \in A \cr t,t_1,t_2\in \QR\end{array} \right\} \ \Bigr\rangle$ 
 is normal in $H$. More precisely, we prove that $H_1$ is stable under conjugacy by any $R_c$ with $c\in A$ and by any $T\in \QR$. Indeed 
 
\medskip
 
$\pbul$ $\dis R_c [R_a,t] R_c ^{-1} = R_c R_a t R_a ^{-1} \ t^{-1} R_c ^{-1} = \underbrace{R_c R_a t R_a ^{-1} \ \boldsymbol{R_c^{-1} t^{-1} }}_{[R_c R_a, t ]}  \underbrace{\boldsymbol{\ t R_c} \ t^{-1} R_c ^{-1}}_{[t,  R_c ]} \in H_1$, 
 
$\pbul$ $\dis R_c [t_1,t_2] R_c^{-1}= \bigl[R_c,[t_1,t_2]\bigr] \ [t_1,t_2]\in H_1$ by Properties \ref{CommPtGene},

\medskip
 
$\pbul$ $\dis T  [R_a,t] T^{-1} = T  R_a \ t R_a ^{-1} \ t^{-1} T^{-1}
=\underbrace{T  R_a \boldsymbol{\ T^{-1} R_a^{-1}}}_{\dis [T,  R_a]}  \ \underbrace{\boldsymbol{\ R_a T} \  t R_a ^{-1} \ t^{-1} T^{-1} }_{\dis [R_a,  Tt ]} \in H_1$,

$\pbul$ and finally, $T[t_1,t_2] T^{-1}=[Tt_1T^{-1},Tt_2T^{-1}]\in H_1$.

\medskip
 
$-$ We now prove that $[H,H]=H_1$. 

It is folklore (see Appendix \ref{geneDGFolk}) that the derived  subgroup of a group $\Gamma=\langle S \rangle$ is the normal closure, $\ll [a,b], \ a,b \in S \gg_{\Gamma}$, of the group generated by the commutators $[a,b]$ with $a,b \in S$. Thus, as $H=\langle \ R_a,\  t \ \ ; \ a\in A, \ t\in Q \ \rangle$, we have 
$$\boldsymbol{[H,H]=\ll [R_a,t] , [t_1,t_2]} \ ; \ a \in A , \  t,t_1,t_2 \in \QR  \boldsymbol{ \gg_H=H_1}$$ since $H_1$ is normal in $H$.
\end{enumerate}
\end{proof}

\section{Proofs of Theorem \ref{THGenF1}}

Let $A=\langle \alpha_1, \cdots, \alpha_s\rangle$,  $\QR=\langle  g_1, ...,g_m \rangle$ and $H=\HA$ as in Definition \ref{defH}.

\begin{proof}[\textbf{Proof of Item (1).} \quad ] 

$H$ is abelian $\Lra\limits^{(1)}$ $\forall i \in \{1,\dots, m\}$, $g_i\in \mathcal G_{\frac{1}{q}}$ is a rational rotation $\LRA$  $\forall i \in \{1,\dots, m\}$, $\exists p_i\in \{0, \dots ,q-1\}$ such that $g_i=R_{\frac{p_i}{q}}$  $\LRA$  $\forall i=1,\dots,m$, $\exists p_i\in \{0, \dots ,q-1\}$  such that $\mathfrak S(g_i)= \sigma^{p_i}$ $\Lra\limits^{(2)}$
$W=\langle  \mathfrak S(Q),\sigma \rangle$ is abelian.

$\boxed{\Ra\limits^{(1)}}$ follows from Novak's result (see \cite{No} Lemma 5.1 and its proof) which states that the centralizer in $\IET$ of an irrational rotation consists of rotations.

$\boxed{\La\limits^{(2)}}$ is a consequence of the fact that the centralizer of a $q$-cycle $\sigma$ in $\mathcal S_q$ is $\langle \sigma \rangle$. 
\end{proof}

\medskip 

\begin{proof}[\textbf{Proof of Item (2).} ]
Since $H/[H,H]$ is abelian and $[H,H] <\ker \ell$ is the direct union of its finitely generated subgroups that are finite by Lemma \ref{L2EM} (4), the group $H$ is obtained from groups that are either finite or abelians by taking one direct union and one  group extension, this exactly means that $H$ is elementary amenable and that it belongs to the Chou class $EG_2$ (see \cite[Section 2]{Chou}).

\medskip

In addition, as $s\not=0$, the group $H$ contains an irrational rotation, and as $\QR\setminus \mathbb S^1 \not=\emptyset$ it also contains an \iet \ that is not a rotation. Therefore, it is not virtually nilpotent by Criterium \ref{CritGLtg}. Moreover, it was proven in \cite[Theorems 3.2 and 3.2']{Chou} that an elementary amenable finitely generated group that is not virtually nilpotent contains a free semigroup on two generators and then it has exponential growth.\end{proof}

\section{Local permutations}\label{perlocJump}

\subsection{Definitions and basic properties}

\begin{defi} 
Let $f\in \ker \ell$, $x\in [0, \frac{1}{q})$ and $\beta \in (0, \frac{1}{q})$ such that $f$ is continuous on the intervals $J_p(x, \beta) = [x+\frac{p-1}{q}, x+\frac{p-1}{q}+\beta)$ for all $p=1, \cdots, q$.

As the translations of $f$ belong to $\{\frac{j}{q},\ j\in \Z\}$, the map $f$ permutes the $J_p(x, \beta)$ and then there exists $\omega\in \mathcal S_q$ such that $$f(J_p(x,\beta))= J_{\omega (p)}(x,\beta)= [x+\frac{\omega (p)-1}{q},\  x+\frac{\omega (p)-1}{q}+\beta)$$
It is easy to see that $\omega$ does not depend on the choice of a suitable $\beta$ and it is denoted by $\omega (f,x)$ and called \textbf{the local permutation of $f$ at $x$}.
\end{defi}

\medskip

A direct consequence of this definition is the following

\begin{proper} \label{formula f w} Let $f\in \ker \ell$, $x\in [0, \frac{1}{q})$ and $i\in \{1,\dots,q \}$ then 
$$\boxed{\boldsymbol{ f\bigl(x+\frac{i-1}{q}\bigr)=x+\frac{\omega(f,x) \bigl(i\bigr)-1}{q}}}$$
In other words, denoting by $x_i=x+\frac{i-1}{q}$, we have $\boldsymbol{\dis f(x_i) = x_{\omega (f,x)(i)}}$.
\end{proper}

\begin{prop} \label{WMorph} Let $x\in [0,\frac{1}{q})$.
The map $\omega_x : \left\{ \begin{array}{ll}
\ker \ell &\to \ \ \mathcal S_q \cr f &\mapsto \ \ \omega (f,x) \end{array} \right.$ is a morphism and its image contains $\mathfrak S(\QR)$. More specifically, if $f\in \QR$ then ${\omega (f,x)=\mathfrak S(f)}$.
\end{prop}

\begin{proof} Let $x\in [0, \frac{1}{q})$, it is obvious that $\omega(\id,x)= \id$. Let $f,h \in \ker \ell$ and $\beta \in (0, \frac{1}{q})$ small enough so that $f$ and $h$ are continuous on the intervals $J_p(x, \beta)$. 
Hence, for any $p\in\{1, \dots, q\}$, we have $$h\circ f (J_p(x, \beta)) = h\bigl(J_{\omega(f,x)(p)}(x, \beta) \bigr) = J_{\omega(h,x)\omega(f,x)(p)}(x, \beta),$$ therefore $\omega (h\circ f,x)= \omega(h,x)\ \omega(f,x)$.

\medskip
It is easy to check that the $\iet$ $t\in \mathcal G_{\frac{1}{q}}$ associated to $\tau \in \mathcal S_q$ satisfies $\omega (t,x)= \tau$ and then we get $\omega_x(\ker \ell) \supseteq {\mathfrak S(\QR)}$.
\end{proof}

\begin{prop}\label{critere Id et abelien} \ 
\begin{enumerate}
\item If $f\in \ker \ell$ satisfies $\omega(f,x)=\id $ for all $x\in [0,\frac{1}{q})$, then $f=\id$.
\item Let $\Gamma < \ker \ell$,  $\Gamma $ is abelian if and only if \ $\forall x\in [0,\frac{1}{q} )$, $\omega_x(\Gamma )$ is abelian.
\end{enumerate}
\end{prop}

\begin{proof}  \ 
 
\begin{enumerate}
\item Let $f\in \ker \ell$ and $x\in [0, \frac{1}{q})$. By Property \ref{formula f w}, if $\omega(f,x)=\id$ then $f(x+ \frac{i-1}{q})= x+ \frac{i-1}{q}$, for any $i\in \{1,\dots,q \}$.

This implies that $f=\id$ provided that $\omega_x=\id$, $\forall x \in [0, \frac{1}{q})$.

\medskip

\item The direct implication is obvious. For the converse, we apply Item (1) to $[f,g]$ with $f,g \in \Gamma $ and any $x\in [0,\frac{1}{q})$. Indeed, we have that $\omega([f,g],x)=[\omega(f,x), \omega(g,x)])= \id $ since $\omega_x(\Gamma )$ is abelian,  and we conclude by Item (1) that $[f,g]= \id$.
\end{enumerate} \ \vskip -4mm \hfill\end{proof}

\begin{lemm} \label{perm loc conjugate commutator} \ 

Let  $x\in [0, \frac{1}{q})$, $f\in \QR$ and $\tau =\mathfrak S(f)= \omega\bigl(f,x\bigr)$. We recall that  $\sigma$ is the $q$-cycle $(1,2,...,q) \in \mathcal S_q$.

Let $\gamma \in (0,1)$ be an irrational number, $\tilde \gamma \in (0, \frac{1}{q})$ and $j_0\in \{0,\dots,q-1 \}$ be defined by $\gamma= \frac{j_0}{q} + \tilde \gamma$. Then

\begin{enumerate}
\item
\begin{enumerate}
\item If $x \in [0, \tilde \gamma)$ then $\dis \omega (R_\gamma f R_{-\gamma }, x) =\sigma^{j_0+1} \tau  \sigma^{-(j_0+1)}$.
\item If $x \in [\tilde \gamma, \frac{1}{q})$ then $\dis \omega (R_\gamma f R_{-\gamma }, x)=\sigma^{j_0} \tau  \sigma^{-j_0}$.
\end{enumerate}

\medskip

\item
\begin{enumerate}
\item If $x \in [0,\tilde \gamma)$ then $\omega ([R_\gamma,f],x)=[\sigma^{j_0+1},\tau] $.
\item If $x \in [\tilde \gamma, \frac{1}{q})$ then $\omega ([R_\gamma,f],x)=[\sigma^{j_0} , \tau]$.
\end{enumerate}
\end{enumerate}
\end{lemm}

\medskip

\begin{proof} In the following calculus, we argue in $\IET(\R/\Z)$ and we consider integers modulo $q$.

\begin{enumerate}
\item 
\begin{enumerate}

\item Let  $x \in [0, \tilde \gamma)$ and $i\in \{1, ..., q\}$. We have  $$\dis x+ \frac{i-1}{q}\in \bigl[\frac{i-1}{q}, \frac{i-1}{q}+ \tilde \gamma\bigr)$$

\hskip -1cm  then \ 
 $\dis R_{-\gamma} (x+ \frac{i-1}{q}) = x-\tilde \gamma + \frac{i-1-j_0}{q} \in  \bigl[\frac{i-1-j_0}{q} -\tilde \gamma, \frac{i-1-j_0}{q}\bigr) \subset I_{i-1-j_0}$,
 
\hskip -1cm therefore $\dis f R_{-\gamma} (x+ \frac{i-1}{q}) = x-\tilde \gamma + \frac{\tau(i-1-j_0)}{q} \in  f(I_{i-1-j_0})= I_{\tau(i-1-j_0)}$


\hskip -1cm  and finally \ $\dis R_{\gamma}   f R_{-\gamma} (x+ \frac{i-1}{q}) = x + \frac{\tau(i-1-j_0)}{q} + \frac{j_0}{q}$. 

\medskip

We conclude that \ $\dis \omega (R_\gamma f R_{-\gamma }, x)(i)-1= \tau(i-1-j_0)+j_0$, thus
$$\dis \omega (R_\gamma f R_{-\gamma }, x)(i)= \tau(i-(j_0+1))+(j_0+1)=\sigma^{j_0+1} \tau  \sigma^{-(j_0+1)}(i).$$

\bigskip

\item Let $x \in [\tilde \gamma, \frac{1}{q})$ and $i\in \{1, ..., q\}$. Similarly, we have  $$\dis x+ \frac{i-1}{q}\in \bigl[\frac{i-1}{q}+\tilde \gamma, \frac{i}{q} \bigr),$$

\hskip -1cm  then \quad 
 $\dis R_{-\gamma} (x+ \frac{i-1}{q}) = x-\tilde \gamma+ \frac{i-1-j_0}{q} \in  \bigl[\frac{i-1-j_0}{q}, \frac{i-j_0}{q}- \tilde \gamma\bigr) \subset \bigl[\frac{i-1-j_0}{q}, \frac{i-j_0}{q}\bigr)$,
 
\hskip -1cm therefore $\dis f R_{-\gamma} (x+ \frac{i-1}{q}) = x- \tilde \gamma + \frac{\tau(i-j_0)-1}{q} \in  \bigl[\frac{\tau(i-j_0)-1}{q}, \frac{\tau(i-j_0)}{q} -\tilde \gamma \bigr)$,

\hskip -1cm and finally  \ $\dis R_{\gamma} f R_{-\gamma} (x+ \frac{i-1}{q}) = x + \frac{\tau(i-j_0)-1}{q} + \frac{j_0}{q}$. 

\medskip

We conclude that $\dis \omega (R_\gamma f R_{-\gamma }, x)(i)-1= \tau(i-j_0)-1+j_0$, thus  $$\dis\omega (R_\gamma f R_{-\gamma }, x)(i)= \tau(i-j_0)+ j_0=\sigma^{j_0} \tau  \sigma^{-j_0}(i).$$
\end{enumerate}

\bigskip

\item  As $\omega_x$ is a morphism, it holds that 
$$ \omega_x([R_\gamma,f],x)=\omega_x(R_\gamma f R_\gamma^{-1}  f^{-1} ,x)= 
\omega_x(R_\gamma f R_\gamma^{-1} ,x) \underbrace{\omega_x(f^{-1} ,x)}_{\dis \tau ^{-1}}.$$
Hence, Item (1) implies the required conclusion.
\end{enumerate}
\end{proof}

\subsection{Images of $\ker \ell$ and $[H,H]$ by the local permutations.} \  

\begin{sublemm} Let $x\in [0,\frac{1}{q})$ and $\tau,\tau_1,\tau_2 \in \mathfrak S(\QR)$. 
Then for any $p\in \{0,\dots,q-1\}$, we have 
$$\sigma ^p \tau \sigma ^{-p} \in \omega_x(\ker \ell)  \ \text{ and } \ \ [\tau_1,\tau_2], \  [\sigma ^p, \tau ]\in \omega_x([H,H]).$$
\end{sublemm}

\begin{quote}
Indeed, Let $p_0\in \{0,...,q-1\}$, $f\in \QR$ such that $\tau=\mathfrak S(f)$ and $a\in A\setminus\{0\}$.

By the density in $[0,1)$ of the set $\bigl\{  \{na\}, \ n \in \N \bigr\}$, where  $\{\ . \  \}$ denotes the fractional part, there exists some integer $m$ such that $ \{ma \} \in [\frac{p_0}{q}, \frac{p_0}{q}+x]$. 

Setting $\gamma = \{ma \} \in (\frac{p_0}{q} ,\frac{p_0+1}{q} )$ and $\tilde\gamma =\gamma -\frac{p_0}{q}$. We get that $x>\tilde\gamma$ and then we are in the situation of Lemma \ref{perm loc conjugate commutator} Item 1 (b). Therefore, 
$\omega_x(R_\gamma f R_{-\gamma})=\sigma^{p_0} \tau \sigma^{-p_0} \in \omega_x(\ker \ell)$.
In addition: 

If $\tau_1,\tau_2 \in \mathfrak S(\QR)$, we have  $[\tau_1,\tau_2]=
\omega_x([E_{\tau_1},E_{\tau_2}]) \in  \omega_x([H,H])$.

Moreover, $\omega_x([R_\gamma,f])\in \omega_x([H,H])$ and  by Lemma \ref{perm loc conjugate commutator} Item 2 (b), we get $\omega_x([R_\gamma,f])= [ \sigma^{p_0},\tau ]$.
\end{quote}

\medskip
This sublemma combined with Lemma \ref{L3EM} leads to  
\begin{lemm}\label{gene wx(kerl)} \ 
\begin{itemize}
\item  The finite group $\omega_x(\ker \ell) =\bigl\langle \ \sigma^p \tau \sigma ^{-p} \ ; \ p\in\{0, \dots , q-1\}, \ \tau \in \mathfrak S(\QR)\  \bigr\rangle$.
 
\item The finite group $\omega_x([H,H]) =\bigl\langle \ [\tau_1,\tau_2] \ , \ [\sigma ^p, \tau] \ ; \ p\in\{0, \dots , q-1\}, \ \tau, \tau_1, \tau_2 \in \mathfrak S(\QR) \ \bigr\rangle$.

\item In particular, $\omega_x(\ker \ell)$ and $\omega_x([H,H])$ do not depend on $x$. 
\end{itemize}
\end{lemm}

\begin{proof} 

By Item  (1) of Lemma \ref{L3EM}, we have $\dis \omega_x(\ker \ell) =\langle \omega_x (R_{a} g R_{-a}) \ ; \ a\in A, \  g\in \QR \ \rangle$ and 

\noindent $\dis \omega_x([H,H])=\langle \omega_x([R_a,t]) \ , \ \omega_x([t_1,t_2]) \ \ ; \ \   a \in A, \   t,t_1,t_2\in \QR \ \rangle$. 
Then, by Lemma \ref{perm loc conjugate commutator}, we get  $$\omega_x(\ker \ell) \ \subset \ \langle \ \sigma ^p \tau \sigma ^{-p} \ ; \ \ p\in\{0, ..., q-1\}, \ \tau \in \mathfrak S (\QR)\ \rangle \ \ \text{ and  }$$
$$\omega_x([H,H])\subset \langle \ [\sigma ^p, \tau]\ , \ [\tau_1,\tau_2] \ ; \ p\in\{0,\dots, q-1\}, \ \tau, \tau_1, \tau_2 \in \mathfrak S(\QR) \ \rangle.$$

\medskip

The reverse inclusion is provided by the previous sublemma. \end{proof}

\section{Proof of Theorem \ref{THF0}}

\subsection{A criterium of non virtual solvability}
\begin{prop}\label{KEY}
If $\mathcal A_q< \omega_x(\ker \ell)$ for all $x\in [0, \frac{1}{q})$, then $\ker \ell$ is not virtually solvable. 
\end{prop}

\begin{proof}
By contradiction, suppose that there exists $K<\ker \ell$ of finite index and solvable. Since a subgroup of finite index always contains a normal subgroup  also of finite index, we may suppose that $K$ is normal in $\ker \ell$. 

\medskip

As $K$ has finite index and $\ker \ell$ is infinite, it holds that $K\not=\{\id\}$. 
Then there exists $x\in [0, \frac{1}{q})$ such that $\omega_x(K)$ is not trivial, according to Proposition \ref{critere Id et abelien} (1).

\medskip

We state that the group $\omega_x(K) \cap \mathcal A_q$ is not trivial. 

\begin{quote}

\ \ \ Since otherwise, by the property that squares of elements of $\mathcal S_q$ belong to $\mathcal A_q$, the group $\omega_x(K)$ would consist in involutions. In particular, $\omega_x(K)$ being not trivial, there exists a non trivial involution $i \in \omega_x(K)$.

\smallskip

As $\omega_x(K)$ is normal in $\omega_x(\ker \ell)$ and $\omega_x(\ker \ell)$ contains $\mathcal A_q$,  the group \break \hfill $\dis \Gamma:= \boldsymbol{\ll itit^{-1}, \ t \in \mathcal A_q \gg_{\mathcal A_q }}$ is contained in $\omega_x(K)$. 

In addition, since  $itit^{-1}\in  \mathcal A_q $, the group $\Gamma$ is a normal subgroup of the simple group $\mathcal A_q $, thus it either coincides with $\mathcal A_q$ or it is trivial. 

\begin{itemize}

\item  If $\Gamma=\mathcal A_q$ then $\mathcal A_q < \omega_x(K)$ which contradicts 
the fact that $\omega_x(K)$ consists of involutions. 

\smallskip

\item If $\Gamma$ is trivial then  $i$ commutes with any $t \in \mathcal A_q $, which is impossible.
\begin{quote}
{\small Indeed, 
$i$ can be written as a product of permutations with disjoints support: $i=(a_1,a_2)(a_3,a_4)...(a_m, a_{m+1})$. And, 
it is easy to check that the $3$-cycle $(a_1,a_2,a_3)\in \mathcal A_q$ ($q\geq 5$) does not commute with $i$.}
\end{quote}
\end{itemize}
\end{quote}

\smallskip

In conclusion, $\omega_x(K) \cap \mathcal A_q$ is a non trivial normal subgroup of $ \mathcal A_q$ which is simple, then $\omega_x(K) \cap \mathcal A_q = \mathcal A_q$. 
Finally, $\omega_x(K) \supset\mathcal A_q$ is not solvable, a contradiction since $\omega_x(K)$ is solvable as image of a solvable group.
\end{proof}

\subsection{Proof of Item (1) of Theorem \ref{THF0}}

In this subsection, we first develop similar arguments to those used for getting a suitable generating set for $[H,H]$, in order to provide such sets for the group $[W,W]$, where $W:=\langle \sigma, \mathfrak S(\QR) \rangle$. The group $W$ plays, in some sense, the role of $\omega_x(H)$. But this can not be directly obtained from the generating set of $[H,H]$, according to the following

\begin{clai} Let $x\in [0, \frac{1}{q})$, the morphism $\omega_x : \ker \ell \to \mathcal S_q$ can not be extended to a morphism from $H$ to $\mathcal S_q$, except in the situation where $H$ is abelian.
\end{clai}

Indeed, let $x\in [0, \frac{1}{q})$ and $\gamma \in (0, \frac{1}{2q})$ such that $\gamma<x<2\gamma$. Let $r\in \QR$, according to these choices, Item (1) of Lemma \ref{perm loc conjugate commutator} applies twice to $x$ and $f=r$. More precisely, 
\begin{quote}

$\pbul$ its subitem  $(b)$ with  $x>\tilde \gamma=\gamma$  ($j_0=0$) implies that $\omega_x( R_\gamma \ r R_{-\gamma}) = \mathfrak S(r)$, and 

\noindent  $\pbul$ its subitem $(a)$ with $x<\widetilde{2\gamma}=2\gamma$  ($j_0=0$) leads to $\omega_x( R_{2\gamma} \ r R_{-2\gamma} ) = \sigma \mathfrak S(r)\sigma^{-1}$.
\end{quote}
\smallskip

Suppose that $\omega_x$ is a morphism defined on $H$, then:

$-$ On one side, $\mathfrak S(r)=\omega_x( R_\gamma \ r R_{-\gamma}) = \omega_x( R_\gamma) \ \omega_x(r) \omega_x(R_{-\gamma})=\omega_x( R_\gamma) \ \mathfrak S(r) \omega_x(R_{\gamma})^{-1}$, this  means that $\omega_x( R_\gamma) $ commutes with $\mathfrak S(r)$. 

$-$ On the other side, $\sigma \mathfrak S(r)\sigma^{-1}=\omega_x( R_{2\gamma} \ r R_{-2\gamma} ) =  \bigl(\omega_x(R_{\gamma})\bigr)^2 \ \mathfrak S(r) \  \bigl
(\omega_x( R_{\gamma})\bigr) ^{-2}$.

\smallskip

Combining these two facts, we get that $\dis \sigma \mathfrak S(r)\sigma^{-1}=\mathfrak S(r)$.

Therefore, for any $r\in \QR$, it holds that $ \mathfrak S(r)$ commutes with the $q$-cycle $\sigma$ so $ \mathfrak S(r)= \sigma^p$ for some $p$ and then $\QR\subset \mathbb S^1$. Finally, Theorem \ref{THGenF1} (1) implies that $H$ is abelian.

\begin{lemm} \label{gene [W,W]} 

Let $W=\langle \sigma, \mathfrak S(\QR) \rangle$ \ and \ $W_0=\Bigl\langle \ [\sigma^p, t] \ , \ [t_1, t_2] \quad ; \quad \left\{\begin{array}{ll}
p \in \{0, \cdots , q-1\} \cr t,t_1,t_2 \ \in \ \mathfrak S(\QR)
\end{array} \right\} \Bigr\rangle$, 
$$\text{ then } \ \ [W,W] =W_0$$
\end{lemm}

\begin{proof} \ 

$-$ We first check that $W_0$ is normal in $W$, that is $W_0$ is stable under conjugacy by $\sigma$ and by any $T \in \mathfrak S(\QR)$. Indeed, as in the proof of Lemma \ref{L3EM}, we have

\begin{itemize}
\item $\dis \sigma  [\sigma ^p ,t ] \sigma  ^{-1}  = \ 
\sigma  \sigma ^p t \sigma ^{-p} {t^{-1}}\sigma  ^{-1}   = \ $
$\dis \sigma ^{p+1} t \sigma ^{-(p+1)} \sigma t^{-1} \sigma^{-1}  = \
\sigma ^{p+1} t \sigma ^{-(p+1)} t^{-1}  \  t  \sigma t^{-1}  \sigma^{-1}$

\hfill $=\ [\sigma ^{p+1} ,t ][t,\sigma ] =[\sigma ^{p+1} ,t ] \ [\sigma, t] ^{-1} \in  W_0$ 

\item $\dis \sigma [t_1,t_2] \sigma^{-1}= \bigl[\sigma,[t_1,t_2]\bigr] \ [t_1,t_2]\in  W_0$, by Properties \ref{CommPtGene},

\item  $\dis T[\sigma ^p ,t ] T^{-1} = T \sigma ^p t \sigma\ ^{-p} t^{-1}  T^{-1}
 =  \underbrace{ T \sigma ^p \ \ \boldsymbol{T^{-1} \sigma^{-p} }}_{ \dis [ T,  \sigma ^p] } \  \underbrace{ \boldsymbol{\sigma ^{p} T} \  t \sigma ^{-p} t^{-1}  T^{-1} }_ {\dis   [\sigma ^{p}, Tt]}\in W_0,$
 
\item $\dis T[t_1, t_2] T ^{-1}= [T t_1  T^{-1}, T t_2 T^{-1}]\in {W_0}$.
 
\end{itemize}

$-$ We now prove that $\dis [W,W]=W_0$.  Due to the folkloric description for  the generators of the derived subgroup (See  Appendix),  it is plain that $$\dis\boldsymbol{ [W,W]=\ll  [\sigma^p, t] \ , \ [t_1, t_2] \gg_{W}}.$$

\smallskip
Therefore, $\dis [W,W]=\ll W_0  \gg_{W}=W_0$, since $W_0$ is normal in $W$.
\end{proof}

\smallskip
As a corollary of this lemma and Lemmas \ref{gene wx(kerl)}, we get

\begin{conse} \label{consgeneAb2} 
It holds that  $\dis \omega_x ([H, H]) =W_0=[W,W]$.
\end{conse}

\medskip

We turn now on to the proof of Item (1) of Theorem \ref{THF0}. Let $q\geq 5$.

\medskip

By hypothesis $\mathcal A _q < W$, then $W$ is not abelian and $H$ satisfies Item (2) of Theorem \ref{THGenF1}. In addition, by simplicity $\mathcal A _q=[\mathcal A _q ,\mathcal A _q ]<[W, W]=W_0$, due to Lemma \ref{gene [W,W]}.

\smallskip

Let $x\in [0,\frac{1}{q})$, by Consequences \ref{consgeneAb2},  we have $W_0 = \omega_x([H,H]) < \omega_x( \ker \ell)$. 
Therefore $\mathcal A _q < \omega_x( \ker \ell)$ and Proposition \ref{KEY} applies to conclude that $\ker \ell$ and then $H$ are not virtually solvable.

\smallskip

Finally, $H$ is not a linear group, since by Schur Theorem ([Sch11]) any linear torsion group is virtually abelian and $\ker \ell$ is a torsion group which is not virtually solvable so not virtually abelian.

\subsection{Proof of Item (2) of Theorem \ref{THF0}.}

In this section,  we establish the second item of Theorem \ref{THF0}. That is,  we prove that if 
 $H=\HA$ as in Theorem \ref{THF0} and $p\in \N^*$, then

\smallskip

\centerline{\bf $H$ is $p$-solvable if and only if $W=\langle \sigma, \mathfrak S(\QR) \rangle$ is $p$-solvable.}

\begin{proof} Let $p\in \N\setminus\{0\}$, we denote by $D^p (H)$ the $p$-th derived subgroup of $H$ (i.e. $H^0=H$, and the $(k+1)$-derived subgroup of $H$ is given inductively by $D^{k+1} (H)= [D^k (H), D^k (H)]$). It holds that

\medskip

$D^p (H)=\{\id\}$ $\Lra$ $D^{p-1} ([H,H])=\{\id\}$

$\Lra\limits^{(1)}$ 
$\forall x\in [0, \frac{1}{q})$, $\omega_x(D^{p-1} ([H,H]))=\{\id\}$

$\Lra\limits^{(2)}$
$\forall x\in [0,\frac{1}{q})$, $D^{p-1} (\omega_x ([H, H]))=\{\id\}$

$\Lra\limits^{(3)}$
$D^{p-1} ([W,W])=\{\id\}$
$\Lra$  $D^{p} (W)=\{\id\}$

\medskip

{\small $\boxed{\Lra\limits^{(1)}}$} \ holds by Proposition \ref{critere Id et abelien} (1).

\medskip

{\small $\boxed{\Lra\limits^{(2)}}$} \ is true since $\omega_x$ is a morphism on $\ker \ell$.

\medskip

{\small $\boxed{\Lra\limits^{(3)}}$} \ is provided by Consequences \ref{consgeneAb2}.

\bigskip

Finally, as $H$ is $p$-solvable if and only if $\left\{\begin{array}{ll}
D^p (H)=\{\id\} \cr
D^{p-1} (H) \not =\{\id\} 
\end{array}\right\}$,
the previous calculus implies that $H$ is $p$-solvable if and only if  $W$ is $p$-solvable.
\end{proof}

\section{Proofs of Theorem  \ref {THRFin} and Corollaries \ref{CPAb} and \ref{Noniso}} 

\subsection{Proofs of Items (1) and (2) of Theorem  \ref {THRFin}} 

\subsubsection{Torsion $-$ Proof of Item (1) of Theorem  \ref {THRFin}} 

According to Lemma \ref{L2EM} (4), it holds that the torsion elements of $\HA$ form the normal subgroup $\ker \ell$  of $\HA$. In addition, by Lemma \ref{L2EM} (1), it holds that 
$\frac{\HA}{\ker \ell}\simeq \ell(\HA)=A$.

\smallskip

Finally, $H\simeq \ker \ell \rtimes A$ since $\ker \ell\lhd H$, $H=\ker \ell . \ R_A$ by Lemma \ref{L2EM} (2) and $\ker \ell \cap R_A=\{\id\}$ by the choice of  $A$.

 \subsubsection{Abelianization $-$ Proof of Item (2) of Theorem  \ref {THRFin}} \label{Abel2a}

We first prove

\begin{prop} \label{Abelianization}  The abelianization $\dis \frac{H}{[H,H]}$ of $H$ is isomorphic to the group  $\dis {\dis \frac{\ker \ell}{[H,H]}} \times A$.
\end{prop}

We recall that, according to Lemma \ref{L2EM} (2) and (3), $[H,H] {\dis \Sgn } \ker \ell$ and any $h\in H$ can be uniquely written as $h = P_h R_{\ell(h)}$ with $P_h\in  \ker \ell$. More precisely, Proposition \ref{Abelianization} is established by proving  

\begin{lemm}\label{AbelMorph} \ 

The map  $\dis D: \ \left\{\begin{array}{ll} 
{\dis \frac{H}{[H,H]} } \quad \to \ \ \  {\dis \frac{\ker \ell}{[H,H]}} \times A  \cr
 h. [H,H] \ \ \mapsto \ \ \bigl(P_h . [H,H], \  \ell(h) \bigr)
\end{array}\right.$ \ \ 
is well defined and is an isomorphism.
\end{lemm}
\begin{sublemm} \label{P(hh')} Let $h, h_1,h_2 \in H$. Then 
$$P_{h^{-1}} \ = P_{h}^{-1} \ \bigl[  P_h , R_{-\ell(h)}   \bigr] \quad \text{ and } \quad
P_{h_1h_2} = P_{h_1} P_{h_2} \ \bigl[  P_{h_2}^{ -1}, R_{\ell(h_1)}   \bigr].$$
\end{sublemm}

\begin{proof} By the definition of $P_f$, for $f\in H$, it holds that 
$$P_{h^{-1}}  R_{\ell(h^{-1})}= h^{-1}= (P_h R_{\ell(h)})^{-1}= R_{-\ell(h)} P_h^{-1} = P_h^{-1} \ \bigl[  P_h , R_{-\ell(h)} \bigr] R_{-\ell(h)} \text{ and} $$
$$P_{h_1h_2} R_{\ell(h_1h_2)} = h_1 h_2 = P_{h_1} \ \ \underbrace{R_{\ell(h_1) } P_{h_2}}_{ P_{h_2} [P_{h_2}^{ -1}, R_{\ell(h_1)} ] R_{\ell(h_1) } } \ \ R_{\ell(h_2)}=
P_{h_1} P_{h_2} \ \bigl[  P_{h_2}^{ -1}, R_{\ell(h_1)}   \bigr] \  R_{\ell(h_1)} R_{\ell(h_2)}.$$
This implies the required formulas.
\end{proof}

\begin{proof}[Proof of Lemma \ref{AbelMorph}] \ 

The map $D$ is well-defined. 

\begin{quote}
Indeed, Let $h_1, h_2 \in H$ such that $h_2=h_1 C$ with $C\in [H,H]$.  
It holds that $\ell(h_2) =\ell (h_1) + \ell (C) = \ell (h_1) $.  In addition, according to Sublemma \ref{P(hh')} and the fact that $P_f=f$ for any $f \in \ker \ell$, we have  $P_{h_2} =P_{h_1 C} =P_{h_1}  P_C [P_C^{-1}, R_{\ell(h_1)}] = P_{h_1}  C [P_C^{-1}, R_{\ell(h_1)}]$. 
As $C [P_C ^{-1}, R_{\ell(h_1)}]\in [H,H]$, we get that $P_{h_2} \in P_{h_1}.[H,H]$.
\end{quote}

\smallskip

The map $D$ is a morphism, since $\ell$ is a morphism and Sublemma \ref{P(hh')} asserts that the second coordinate is also a morphism.

The  map $D$ is surjective, as for any $a\in A$ and any $f \in \ker \ell$, we have $\ell(R_a f) =a$ and $P_{R_a f}=f$.

Finally, $D$ is injective since $\ell(h_2) =\ell (h_1) $ and  $P_{h_2} =P_{h_1} C$ with $C\in [H,H]$  imply that $h_1^{-1}h_2 = P_{h_1^{-1}h_2}= P_{h_1^{-1}}P_{h_2 } C'$ by Sublemma \ref{P(hh')}. Then $h_1^{-1}h_2 =CC'\in [H,H]$.
\end{proof}

\begin{prop}\label{FinitePart}
The map $\rho : \left\{\begin{array}{ll} 
\mathfrak S(Q) \to  \  \ \frac{\ker \ell}{[H,H]} \cr
\ \  \tau \ \  \ \mapsto  E_\tau . [H,H]
\end{array} \right.$,  where $E_\tau\in \mathcal G_{1/q}$ is defined by $\mathfrak S(E_\tau)=\tau$,  
is a surjective morphism whose kernel satisfies $[\mathfrak S(\QR),\mathfrak S(\QR)] <\ker \rho< [W, W] \cap \mathfrak S(\QR)$.
\end{prop}

\begin{proof} 

It is clear that $\rho$ is a morphism and its surjectivity will be provided by 

\begin{claim}\label{claimab}
Any $f\in \ker \ell$ can be written as $f= r_f C_f$ where  $r_f\in \QR$ and $C_f \in [H,H]$.
\end{claim}

\begin{quote}
Indeed, we argue by induction on the minimal number $m$ such that $f$ can be written as a product of $m$ elements of the form $R_{a_i} r_i R_{-a_i}$ with $a_i\in A$ and $r_i\in \QR$ (existence is provided by Lemma \ref{L3EM} (1)).

\smallskip

If $m=0$ then $f=\id\in \QR$ and the claim is obvious.

\smallskip

Let $m\in \mathbb N$, we suppose that any element of $\ker \ell$ which is a product of $m$ elements of the form $R_{a_i} r_i R_{-a_i}$ can be written as $r C$ with $r\in \QR$ and $C\in [H,H]$. Let $f\in \ker \ell$ a product of $m+1$ elements $R_{a_i} r_i R_{-a_i}$. We have 

$\dis f = \prod_{i=1}^{m+1} R_{a_i} r_i R_{-a_i} =  \prod_{i=2}^{m+1} R_{a_i} r_i R_{-a_i} \  \ R_{a_1} r_1 R_{-a_1} $, and by induction hypothesis, 

$\dis \prod_{i=2}^{m+1} R_{a_i} r_i R_{-a_i}={\dis r C}$ with $r\in \QR$ and $C \in [H,H]$. Therefore 

$\dis f =r C r_1 [r_1^{-1}, R_{a_1}] = \underbrace{r r_1 }_{\in \QR} \underbrace{ C [C^{-1}, r_1^{-1}] \ [r_1^{-1}, R_{a_1}]}_{\in  [H,H]}$ has the required form.
\end{quote}

\smallskip

We turn now on to the proof of Proposition \ref{FinitePart}.

The previous claim means that any class modulo $[H,H]$ contains an element of $\QR$ and then gives rise to the surjectivity of $\rho$.

\medskip

In addition, $\tau \in  \ker  \rho$ if and only if $E_\tau \in  \QR \cap [H,H]$. In particular, since $\QR \cap [H,H]$ contains $[\QR,\QR]$, we have  $[\mathfrak S(\QR),\mathfrak S(\QR)] <\ker \rho$. Moreover, if $r \in\QR \cap [H,H]$ then $\mathfrak S(r)=\omega_0 (r) \in
\omega_0 (\QR \cap [H,H])= \mathfrak S(Q)\cap[W,W]$, by Proposition \ref{WMorph}  and Consequences \ref{consgeneAb2}.

This gives the other inclusion and concludes the proof of Proposition \ref{FinitePart}.
\end{proof}

We turn now on to the proof of Item (2) of Theorem \ref{THRFin}.

\smallskip

It is obvious that $\frac{\ker \ell}{[H,H]} $ is an abelian group. 
 
The surjectivity of $\rho$ and the first isomorphism theorem imply that $ \frac{\ker \ell}{[H,H]} \simeq \im  \rho \simeq \frac{\mathfrak S(Q)}{\ker\rho}$. 

Then combining Propositions \ref{Abelianization} and \ref{FinitePart},  we get Item (2) of Theorem \ref{THRFin}.

\begin{rem} When the two groups involved in Proposition \ref{FinitePart} ( $[\mathfrak S(\QR),\mathfrak S(\QR)]$ and $ [W,W] \cap \mathfrak S(Q)$) coincide, we can provide the abelianization as in  Corollary \ref{CPAb}. However, they might be distinct and the description of the abelianization requires further studies (see Example \ref{Subtil7.4}).
\end{rem}

\subsection{Proof of Corollaries \ref{CPAb} and \ref{Noniso}} \label{consgeneAb3}  
\subsubsection{Proof of Corollary \ref{CPAb}} \ 

Let $A$, $\QR$ be as in Definition \ref{defH}, $\sigma=\sigma_q$ be the $q$-cycle $(1,2,\cdots,q) \in \mathcal S_q$ and $W=\langle\mathfrak S(Q), \sigma \rangle$. 

\smallskip

\begin{enumerate}
\item If $[W,W] \cap \mathfrak S(Q)=[\mathfrak S(Q),\mathfrak S(Q)]$ then by Item (2) of Theorem \ref{THRFin}, we get $N_Q= [\mathfrak S(Q),\mathfrak S(Q)]$ and therefore $F\simeq \frac{\mathfrak S(Q)}{N_Q}= \frac{\mathfrak S(Q)} {[\mathfrak S(Q),\mathfrak S(Q)]}$.

\smallskip

In particular, if $\sigma\in \mathfrak S(Q)$ then $W=\mathfrak S(Q)$ and $[W,W] =[\mathfrak S(Q),\mathfrak S(Q)]$. Hence we are in the previous situation.

\medskip

\item If $\mathfrak S(Q)=\langle \tau \rangle$ with $\tau \in \mathcal S_q\setminus \mathcal A_q$ then $[\mathfrak S(Q),\mathfrak S(Q)]=\{\id\} $ and $ [W,W] \cap \mathfrak S(Q)=\{\id\} $ since $\tau \notin \mathcal A_q$. 
Hence,  by Item (2) of Theorem \ref{THRFin}, we have  $N_Q= \{\id\} $ and then $F\simeq \mathfrak S(Q)$.

\medskip

\item If $\mathfrak S(Q)=[\mathfrak S(Q),\mathfrak S(Q)]$, then the group $[W,W]$ contains $\mathfrak S(Q)$ and  $ [W,W] \cap \mathfrak S(Q)=\mathfrak S(Q)=[\mathfrak S(Q),\mathfrak S(Q)]$. Hence,  by Item (2) of Theorem \ref{THRFin}, $N_Q= \mathfrak S(Q)$ and  $F= \{\id\}$.

\smallskip

In particular, if $\mathfrak S(Q)= \mathcal A_q$ ($q\geq 5$) then $\mathfrak S(Q)=[\mathfrak S(Q),\mathfrak S(Q)]$ and we get $F=\{\id\}$.
\end{enumerate}

\subsubsection{Proof of Corollary \ref{Noniso}} \ 
\begin{enumerate}
\item This item follows from the fact that if $A\not\simeq A'$ then the groups $\frac{\HA}{T(\HA)}$ and $\frac{\HAp}{T(\HAp)}$ are not isomorphic.

\medskip

\item This item is due to the invariance by $\iet$-conjugacy of the $\saf$-invariant. For backgrounds on $\saf$, we refer the reader to \cite[Section 2]{Bos}.

Indeed, for any $g\in H$, we have $g= R_{a} P$, where $a\in A$ and $P$ of finite order (by Lemma \ref{L2EM} (3)) then $\saf(g)=\saf(R_a) = a\wedge 1$. Therefore $\saf(H) =\{ a\wedge 1, a \in A\}$.
This implies that if $A\not= A'$ then $\saf(H)\not=\saf(H')$.

\medskip

\item Suppose that $\sigma_q \in  \mathfrak S(\QR)$ and $\sigma_{q'} \in  \mathfrak S(\QR')$ then $W= \mathfrak S(\QR)$ and $W'= \mathfrak S(\QR')$.  Therefore, according to Corollary \ref{CPAb},  $F\simeq \mathfrak S(\QR)_{\ab}$ and $F'\simeq \mathfrak S(\QR')_{\ab}$.  In particular, as $\QR$ and $\QR'$ have non isomorphic abelianizations, $\HA$ and $\HAp$ neither ; so they are not isomorphic. 
\end{enumerate}

\section{Comparison with lamplighter groups and examples}
In this section, we are interested in describing conditions under which the group $H_{A,Q}$ is isomorphic to a Lamplighter group  $L\wr G$ where $L$, $G$ are non trivial finitely generated groups and $G$ is not a torsion group.

\subsection{Break points.}

\begin{defi} Let $f$ be an iet seen as an element of $\IET(\R/\Z)$. 

We denote by $\BP_{\mathbb \R/\mathbb \Z}(f)$  the set consisting of the discontinuity points of $f$ on $\mathbb \R/\mathbb \Z$. 
\end{defi}

\begin{proper} Let $f_1, \cdots, f_n \in \IET(\R/\Z)$. It holds that
$$\BP_{\mathbb \R/\mathbb \Z}(f_n \circ f_{n-1} \cdots  \circ f_1)  \subseteq \BP_{\mathbb \R/\mathbb \Z}(f_1) \cup f_1 ^{-1} \bigl( \BP_{\mathbb \R/\mathbb \Z}(f_2) \bigr) \cdots 
\cup  (f_1 ^{-1} \circ \cdots \circ f_{n-1} ^{-1}) \bigl( \BP_{\mathbb \R/\mathbb \Z}(f_n) \bigr)  $$
and if the sets $\dis (f_1 ^{-1}\circ  \cdots \circ f_{j-1} ^{-1}) \bigl( \BP_{\mathbb \R/\mathbb \Z}(f_j) \bigr)  $ are pairwise disjoint then the previous inclusion is an equality.
\end{proper}

\begin{conse} \label{consBP}Let $A$, $\QR$ be as in Definition \ref{defH}. 

Let $f_1, \cdots, f_n \in \ker \ell$  such that the sets $[\BP_{\mathbb \R/\mathbb \Z}(f_j)]_q$ are pairwise disjoint then $$\bigl[\BP_{\mathbb \R/\mathbb \Z}(f_n \circ f_{n-1} \cdots  \circ f_1) \bigr]_q= \sqcup_j [\BP_{\mathbb \R/\mathbb \Z}(f_j) ]_q$$

where  $[X]_q:=\cup_p R_{p/q} (X)$, for a given $X\subset \R/\Z$.
\end{conse} 

This follows from the previous property and the fact that the set of translations of $f\in \ker \ell$ is contained in $\{p/q, \ p=0, \dots , q-1\}$.

\bigskip

We now give an example of a realization of the classical Lamplighter group as a $\HA$.

\subsection{The lamplighter $\boldsymbol{L_{2,1}=\langle a,t \ | \ a^2=\id, \ [a, t^k a t^{-k}]=\id \rangle}$ as a $\boldsymbol{\HA}$} \label{ClLamp} \ 

\medskip

Let $H=\langle g,R_\alpha\rangle$ with $\alpha\notin \R/\Q$ and $g \in \mathcal G_{\frac{1}{4}}$ represented on Figure 1.

\vskip -3mm

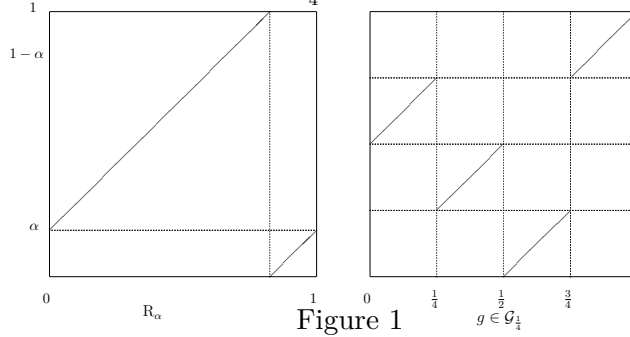
\begin{figure}[h]\scalebox{.5}{
\begin{picture}(450,200)
\put(0,0){\line(1,0){200}} \put(0,200){\line(1,0){200}}
\put(0,0){\line(0,1){200}} \put(200,0){\line(0,1){200}}

\put(-5,-20){$0$}
\put(195,-20){$1$}

\put(-15,200){$1$}

\put(165,0){\dashbox(0,200){}}

\put(0,35){\dashbox(200,0){}}

\put(0,35){\line(1,1){165}}
\put(165,0){\line(1,1){35}}

\put (-15,35){$\alpha$}
\put (-30,165){$1-\alpha$}
\put (70,-30){\large{R}${}_{\alpha}$}

\put(240,0){\line(1,0){200}} \put(240,200){\line(1,0){200}}
\put(240,0){\line(0,1){200}} \put(440,0){\line(0,1){200}}

\put(235,-20){$0$}
\put(285,-20){${\frac{1}{4}}$}
\put(335,-20){${\frac{1}{2}}$}
\put(385,-20){${\frac{3}{4}}$}
\put(435,-20){$1$}

\put(290,0){\dashbox(0,200){}}
\put(340,0){\dashbox(0,200){}}
\put(390,0){\dashbox(0,200){}}
\put(390,0){\dashbox(0,200){}}

\put(240,50){\dashbox(200,0){}}
\put(240,100){\dashbox(200,0){}}
\put(240,150){\dashbox(200,0){}}

\put(240,100){\line(1,1){50}}
\put(290,50){\line(1,1){50}}
\put(340,0){\line(1,1){50}}
\put(390,150){\line(1,1){50}}

\put ( 320,-34){\large{$g\in \mathcal G_{\frac{1}{4}}$}}
\end{picture}

} \end{figure} 
 
\centerline{Figure 1}

\bigskip

Note that $H= H_{A,Q}$ with $A=\langle\alpha \rangle$ and $Q=\langle (1,3) \rangle <\mathcal G_{\frac{1}{4}}$ so that the morphism $\ell$ is well defined on $H$.  We claim that  \textit{the correspondence $\bds{\bigl( a \mapsto \ g}$ ,  $\bds{t \mapsto R _\alpha\bigr)}$ induces a faithful representation of $L_{2,1}$ in $\IET$.}

\smallskip

Indeed, this correspondence  induces a morphism $\rho : L_{2,1}\to \IET$, since 
$g^2=\id$ and   the conjugates of $g$ are pairwise commuting.

\begin{quote} This last property comes from the fact that  $\ker \ell$-orbits have cardinality at most  $2$ (as the translations of $g$ belong to $\{\pm \frac{1}{2}\}$, the set $\ker \ell(x)\subset \{x, x\pm \frac{1}{2}\}$) and then that  any element of $\ker \ell$ is an involution. In particular, $\ker \ell$ is abelian.
\end{quote} 

\medskip

In addition, let $\gamma \in \ker \rho$, we write $\gamma = t^m \prod_j t^{k_j} a t^{-k_j}$. 

Since $\ell(\rho(\gamma))=0$, $g^2=\id$ and the $R_{n\alpha} g R_{-n\alpha}$, $n\in \Z$ commute, it holds that 
$\rho(\gamma)= \prod_j R_{k_j\alpha} g R_{-k_j\alpha}$, where $k_j$ are pairwise distinct integers. 

\medskip

Let $g_j:=R_{k_j\alpha} g R_{-k_j\alpha}$, we have that $g_j \in \ker \ell$ and the sets $[\BP_{\mathbb \R/\mathbb \Z}(g_j) ]_q = R_{k_j\alpha} \bigl( [\BP_{\mathbb \R/\mathbb \Z}]_q(g)\bigr)$ are pairwise disjoint.

\medskip

Hence, using Consequence \ref{consBP}, we get that $$[\BP_{\mathbb \R/\mathbb \Z}]_q (\rho(\gamma))= \sqcup_j [\BP_{\mathbb \R/\mathbb \Z}]_q (R_{k_j\alpha} g R_{-k_j\alpha})=\sqcup_j R_{k_j\alpha} ([\BP_{\mathbb \R/\mathbb \Z}]_q(g))$$

Therefore, $[\BP_{\mathbb \R/\mathbb \Z}]_q (\rho(\gamma))$ is trivial if and only if $\gamma$ is. This proves that $\rho$ is injective.

\begin{rema}\ 

This embedding of $L_{2,1}$ in $\IET$ is slightly different from the one given in \cite[Section 4]{DFG}. 
\end{rema}

\subsection{Proof of of Corollary \ref{NonisoLamp}}

\begin{lemm}\label{LFinAb} \ 

Let $L$, $G$ be finitely generated groups such that $G$ contains an element $g_0$ of infinite order and $\bds{L\wr G\simeq H_{A,Q}}$ for some $A$ and $Q$ (as in Definition \ref{defH}).
\textbf{Then $L$ is a finite abelian group.}
\end{lemm} 

\begin{proof}
Let $L$, $G$ be finitely generated groups such that  $H=L\wr G$ embeds in $\IET$ and $G$ contains an element $g_0$ of infinite order.
Then $L\wr \langle g_0 \rangle $ embeds in $\IET$.  But it is proved  in \cite[Section 4]{DFG} that $\Z \wr \Z$ is not a subgroup of $\IET$, therefore any element of $L$ has finite order.
 
\medskip
 
Moreover, as $L$ is finitely generated  either  $L$ is Burnside group or $L$ is finite, in this last case $L$ is abelian by \cite[Section 4]{DFG}.

Suppose that, in addition, $H=H_{A,Q}$ for some $A$ and $Q$. By Theorem \ref{THGenF1}, $H$ is elementary amenable and then by \cite[Theorem 2.3]{Chou}, it does not contain Burnside groups. 
\end{proof}

\begin{proof}[Proof of Corollary \ref{NonisoLamp}.] We have to prove that : 

\textit{``$H_{A,Q}\simeq L\wr G$ with $L$ finitely generated and $G\simeq \Z^k$ if and only if $V=\langle \sigma^p \mathfrak S(Q)\sigma^{-p}, \ p\in \IN \ \rangle$ is abelian.''}

\medskip

For the direct implication, looking at the torsion groups, we have

\ \quad $T(L\wr G )=\bigoplus\limits_{g\in G} L^g $ \quad and 
\ \quad $T(H_{A,Q}) =\ker \ell$

\smallskip

As $L$ is abelian, $\bigoplus\limits_{g\in G} L^g $ is abelian and we get that $\ker \ell$ is abelian. By Proposition \ref{critere Id et abelien} and  Lemma \ref{gene wx(kerl)}, this is equivalent to  $V=\langle \sigma^p \mathfrak S(Q)\sigma^{-p}, \ p\in \IN \ \rangle$ is abelian.

\bigskip

For the reverse implication, we have that $\HA=\ker \ell \rtimes R_A$ and $\ker \ell =\langle  R_a g  R_a^{-1}, a\in A, g\in Q \rangle$ is abelian by hypothesis, Proposition \ref{critere Id et abelien} and Lemma \ref{gene wx(kerl)}.

\smallskip

We now check that $ R_a g  R_{a}^{-1}$ and $R_{a'} g' R_{a'}^{-1}$ are distinct provide that $a$ and $a'$ are distinct.
\begin{quote} 
Indeed,  $R_a g  R_{a}^{-1}= R_{a'} g' R_{a'}^{-1}$ implies that $g$ and $g'$ are conjugate through an irrational rotation $R_b$. Comparing break point sets, 
 we get $[\BP_{\mathbb \R/\mathbb \Z}]_q (g) \subset [ \{ 0 \} ]_q$ 
and $[\BP_{\mathbb \R/\mathbb \Z}]_q (R_b g R_{-b})\subset [\{b\}]_q$, a contradiction.
\end{quote}
 
Therefore  $\ker \ell = \bigoplus\limits_{a\in A} Q^{R_a}$ and $\HA= \bigoplus\limits_{a\in A} Q^{R_a} \rtimes R_A = Q\wr R_A$.

\medskip

In this case, $V$ is abelian and then $Q$ is abelian. Computing abelianizations, we get: 

\ \  $H_{A,Q} ^{ab} \simeq  F\times A$ with $F=T(H_{A,Q})/[H_{A,Q},H_{A,Q}]$ by Lemmas \ref{L2EM} (4) and \ref{AbelMorph}, 

\ \  $(Q\wr R_A)^{ab} \simeq Q \times A$ \ \  and 
 \ \  $(L\wr G )^{ab} \simeq L \times G$.

\smallskip

This implies that  $A\simeq G \simeq \Z^k$ and $F\simeq Q\simeq L$.
\end{proof}

\begin{rema} For the Lamplighter group $L_{2,1}$, it holds that $W=\langle \sigma, \tau\rangle= \langle (1,2,3,4), (1,3)\rangle = \langle \sigma,\tau \ \vert \ \sigma^4= \tau^2=(\sigma\tau)^2=\id \rangle$ is the diedral group $\D_8$ so it is metabelian 
and $V=\langle\sigma^p \tau \sigma^{-p} \rangle=\langle (1,3), (2,4) \rangle$ is abelian. Therefore Corollary \ref{NonisoLamp} reproves that $H=\langle g, R_\alpha \rangle$ of Subsection \ref{ClLamp} is the Lamplighter group $\frac{\Z}{2Z}\wr \Z$.
\end{rema}

\bigskip

\subsection{Solvable groups.}
\subsubsection{Metabelian groups that are not isomorphic to Lamplighters} \ 

\medskip

We consider the groups studied in \cite[Subsection 7.4]{GLtg}. Namely, $H=\HA$ with $A=\langle \alpha \rangle$, $\alpha \in [0, \frac{1}{3})\setminus \Q$ and $\mathfrak S (Q)=\langle \tau=(1,3) \rangle=\{\id,(1,3) \} < \mathcal S_3$.

\smallskip

We have that $W= \langle \sigma, \tau \rangle= \mathcal S_3$, \ $[W,W]= [\mathcal S_3, \mathcal S_3]=\mathcal A_3$ is abelian, hence by Theorem \ref{THF0} (2), the group $H=\HA$ is metabelian.

Moreover, $V=\langle \sigma^p \tau \sigma^{-p} \rangle$ is not abelian since $[\sigma, \tau]$ has order $3$. So $H=\HA$ is not a Lamplighter group.

\smallskip

We can also note that $H$ and the Lamplighter groups $\frac{\mathbb Z}{2 \mathbb Z}\wr \Z $ have same abelianization $\frac{\mathbb Z}{2 \mathbb Z}\times \Z $, by Corollary \ref{CPAb} (2). This fact was wrongly written in \cite{GLtg}.

\subsubsection{Example of a $3$-solvable group} \label{Subtil7.4}\

\medskip

Let $\mathfrak S (\QR) =\langle (1,2) ; (3,4) \rangle < \mathcal S_4$ and $A=\langle \alpha \rangle$, $\alpha \in \R/{}_{\dis\Z}\setminus \Q/{}_{\dis\Z}$.

\smallskip

We have  $W= \mathcal S_4$, $[W,W]= [\mathcal S_4, \mathcal S_4]=\mathcal A_4$  and $[\mathcal A_4, \mathcal A_4]$ is abelian. Hence, by Theorem \ref{THF0} (2), the group $H=\HA$ is $3$-solvable.
In addition, it holds that 

\noindent $[\mathfrak S (\QR) , \mathfrak S (\QR) ]=\{\id\}$ and $[W,W]\cap\mathfrak S (\QR) = \langle \tau_0\rangle$ where $\tau_0=(1,2)(3,4)$ since $(1,2), \ (3,4) \notin \mathcal A_4$.

Then, Proposition \ref{FinitePart} implies that the map $\dis \rho : \left\{\begin{array}{ll} {\dis \mathfrak S (Q )  } \to \ \  {\dis \frac{\ker \ell}{[H,H]} } \cr  \tau \ \mapsto \ \ E_\tau  [H,H] \end{array}\right.$ \  
is a surjective morphism whose kernel is either trivial or equal to $\langle \tau_0 \rangle$.

\smallskip

Actually, this kernel depends on the fact that $E_{\tau_0}$ belongs to $[H,H]$ or not.

\medskip

Using Lemma \ref{perm loc conjugate commutator}, we can check that if $a\in (1/q, 2/q)$ then 
$g:=[ R_a , E_{\tau_2}] [ R_a , E_{\tau_1}] ^2=E_{\tau_0}$.

\medskip

\begin{quote} Indeed, as $[\BP_{\R/\Z}]_q ([ R_a , E_{\tau_2}] [ R_a , E_{\tau_1}] ^2) \subset [\{0,a\}]_q$, we have to prove that the set $[\{a\}]_q$ does not contain discontinuity point of $g$. For this, we are going to prove that $\omega (g,0) = \omega (g,\tilde a) = \tau_0$, where $\tilde a=a -1/q \in (0,1/q)$:

\medskip

\noindent Lemma \ref{perm loc conjugate commutator} and straightforward calculus lead to 

\smallskip

$\pbul$  $\omega ( [R_a , E_{\tau_1}]^2,0) = [\sigma^2, \tau_1]^2= ((1,2)(3 ,4)) ^2=\id$,

$\pbul$ $\omega ([ R_a , E_{\tau_2}],0) = [\sigma^2, \tau_2]= (1,2)(3,4)=\tau_0$,

\medskip

$\pbul$ $\omega ( [R_a , E_{\tau_1}]^2,\tilde a) = [\sigma, \tau_1]^2= (1,3,2)^2=(1,2,3)$ and

$\pbul$  $\omega ( [ R_a , E_{\tau_2}],\tilde a) =[\sigma, \tau_2]= (1,4,3)$.

\medskip

\noindent Then, $\omega (g,0) =\tau_0$ and $\omega (g,\tilde a) = (1,4,3)(1,2,3)=\tau_0$.
\end{quote}

This proves that $E_{\tau_0}\in [H,H]$ so that  $\ker \rho$ is  trivial, and then the abelianization of $H$ is $(\frac{\mathbb Z}{2 \mathbb Z})^2\times \Z $.

\begin{rema} \label{remnsol}
Similar examples of $3$-solvable groups are provided by taking $\mathfrak S (\QR) =\langle (1,2)  \rangle < \mathcal S_4$ or $\mathfrak S (\QR) = \mathcal S_4$ and $A\simeq \Z^s$, $s\in \N\setminus\{0\}$. This leads to infinitely many $3$-solvable groups in $\IET$. 

In addition, finding solvable $\HA$ with arbitrary derived length $n$ reduces to find $W$ a finite solvable subgroup of $\mathcal S_q$ containing the $q$-cycle $\sigma$ and of derived length $n$. This can be done by iterating semidirect products, as follows.

\end{rema}

\subsubsection{Example of $n$-solvable groups.} \ 

\smallskip

\noindent {\large{\textsf{Construction for any $n\in \N\setminus\{0\}$ of a finite solvable subgroup of some $\mathcal S_q$ containing the $q$-cycle $\sigma$ and of derived length $n$.}}}

\smallskip

Let $\mathcal G_{2} =\mathcal G_{2^1}=\mathcal S_{2}$ abelian ($1$-solvable).


We argue by induction, let $n\in \N_{\geq 1}$ and  $\mathcal G_{2^n}$ be a  subgroup of $\mathcal S_{2^n}$ that is $n$-solvable and contains the ${2^n}$-cycle $\bigsigma_{2^n}$.

We consider two copies of  $\mathcal G_{2^n}$: $\mathcal G_{\{1,2,\cdots, {2^{n}}\}} $ viewed as a subgroup of the permutation group of $\{1,2,\cdots, {2^{n}}\}$ and $\mathcal G_{\{ 2^n+1 , \cdots, {2^{n+1}}\}}$ viewed as a subgroup of the permutation group of $\{ 2^n+1 , \cdots, {2^{n+1}}\}$.

\smallskip

We define a subgroup of $\mathcal S_{2^{n+1}}$  by $$\dis G_{2^{n+1}}= \bigl(\mathcal G_{\{1,2,\cdots, {2^{n}}\}} \oplus \mathcal G_{\{ 2^n+1 , \cdots, {2^{n+1}}\}}\bigr)\  \text{\large{$ \psd\limits_{\tau}$}} \ \Z / 2\Z \simeq \mathcal G_{2^{n}}\wr \Z / 2\Z$$ where $\tau =(\bigsigma_{2^{n+1}})^{2^{n}}$ permutes the sets $\{1,2,\cdots, {2^{n}}\}$ and  $\{2^{n}+1,\cdots ,2^{n+1}\}$. It can  be defined by $\tau (k) =k+ 2^{n} \mod 2^{n+1}$.

\medskip

For $g\in  G_{2^{n+1}}$ we note $\overline{g} = \tau g \tau$. We have $g \tau = \tau \overline{g} $, $\overline{g} \tau = \tau g $, and  for all $ g_1,g_2 \in \mathcal G_{2^{n}}$, the elements $g_1$ and $\overline{g_2}$ commute. In addition, we can write  $\dis G_{2^{n+1}}= \bigl(\mathcal G_{2^n}\oplus \overline{\mathcal G_{2^n}} \bigr) \  \text{\large{$ \psd\limits_{\tau}$}} \ \Z / 2\Z $

\medskip

In particular, $\mathcal G_{4}$ is the diedral group $\D_4$, it is metabelian and it contains a $4$-cycle.

\bigskip

We claim that $G_{2^{n+1}}$ contains an $2^{n+1}$-cycle $s=\tau \Pi_1$, where $\Pi_1=\bigsigma_{2^n}=(1,2,\cdots, 2^{n})$, $\Pi_1(k) = k+1 \mod 2^n$ if $k\in \{1,2,\cdots, {2^{n}}\} $ and  $\Pi_1(k) = k$ otherwise. 

\begin{lemm}
The group  $G_{2^{n+1}}$ is $(n+1)$-solvable.
\end{lemm}

\begin{proof}  \

According to Proposition \ref{derlen} (3), it holds that $n\leq \cl (G_{2^{n+1}}) \leq \cl(\mathcal G_{2^n})+\cl(\Z / 2\Z)=n+1$.

We now prove that  $\cl(G_{2^{n+1}})\geq n+1$.

\medskip

$\bullet$ We first prove that $ \dis \bds{\Delta :=\Bigl\{ g.\overline{w} \ ; \ g\in \mathcal G_{2^{n}}, w \in  g^{-1}  . \ D(\mathcal G_{2^{n}}) \Bigr\}}$ 
 is a subgroup of $D(G_{2^{n+1}})$.

\quad $\pbul$ $\Delta \subset D(G_{2^{n+1}})$ since for any $g.\overline{w} \in \Delta$, it holds that   $g.\overline{w}= g.\overline{g^{-1} . C }=[g,\tau] .\overline{C}$ with $C\in D(\mathcal G_{2^{n}})$.
In addition, it is clear that $\id \in\Delta$.

\quad $\pbul$ $\Delta$ is stable by the law. Indeed, let $ g_1\overline{w_1}, \  g_2\overline{w_2} \in \Delta$ with $g_1, g_2 \in \mathcal G_{2^{n}}$ and  $w_1 = g_1^{-1} . C_1$,  $w_2 = g_2^{-1} . C_2$ where $C_1, C_2\in D(\mathcal G_{2^{n}})$. We have

\qquad $\dis (g_1\overline{w_1})^{-1} =  g_1^{-1} . \overline{w_1^{-1}}$ with $w_1 ^{-1} = C_1 ^{-1}g_1= g_1 [g_1^{-1}, C_1^{-1}] C_1 ^{-1}$ and

\qquad $\dis g_1\overline{w_1} . g_2\overline{w_2}  =g_1 g_2 . \overline{w_1w_2}$ with

$\dis w_1 w_2= g_1^{-1} . C_1  g_2^{-1} . C_2 =   g_1^{-1} g_2^{-1} [g_2,C_1] C_1 C_2
=g_2^{-1} g_1^{-1} [g_1,g_2][g_2,C_1] C_1 C_2 \in (g_1 g_2 )^{-1} . D(\mathcal G_{2^{n}}) $

\medskip

$\bullet$ We easily check that the map 
$\bds{p : \Delta \to \mathcal G_{2^{n}}, \ g.\overline{w} \mapsto g}$, is a surjective morphism.

\medskip

$\bullet$ Finally $p$ induces surjective morphisms on the derived subgroups and we get $p: D^{n-1}(\Delta) \twoheadrightarrow D^{n-1}(\mathcal G_{2^{n}}) \not=\{\id\}$. Therefore 
$D^n(G_{2^{n+1}}) \supseteq D^{n-1}(\Delta) \not=\{\id\}$, then $\cl(G_{2^{n+1}})\geq n+1$.
\end{proof}

\medskip

We finish our construction by considering the conjugate, $\mathcal G_{2^{n+1}}$, of $G_{2^{n+1}}$ by the permutation that conjugates $\tau \Pi_1 $ to $\sigma_{2^{n+1}}$.

\bigskip

\noindent {\large{\textsf{Infinite $n$-solvable groups in $\IET$.}}}

\smallskip

A consequence of the last construction, Remark \ref{remnsol}, Theorem \ref{THF0} (2) and Corollary  \ref{Noniso} (1) is

\begin{prop} \label{nSol} \ 

\begin{enumerate}
\item Let $s,n\in \N\setminus\{0\}$, the group $\HA$ with $A\simeq \Z^s$ and $\mathfrak{S}(Q)=\mathcal G_n$ is $n$-solvable.

\item Given $n\in \N\setminus\{0\}$, there exist infinitely many non pairwise isomorphic $n$-solvable infinite groups in $\IET$. 
\end{enumerate}
\end{prop} 

\subsection{Non virtually solvable groups}  \ 

\smallskip

A consequence of Theorem \ref{THF0} (1) and Corollary \ref{Noniso} (1) is

\begin{prop} \label{nonVSol} \ 

\begin{enumerate}
\item Let $s,q \in \N$, $s\geq 1$ and $q \geq 5$, the group $\HA$ with $A\simeq \Z^s$ and $\mathfrak{S}(Q)\in \{\mathcal S_q, \mathcal A_q\}$ is an infinite finitely generated  non virtually solvable and non linear subgroup of $\IET$.

\item There exist infinitely many non pairwise isomorphic infinite finitely generated non virtually solvable and non linear subgroups of $\IET$.
\end{enumerate}
\end{prop} 

\medskip

\begin{rema} \ 

$\pbul$ Using Theorem \ref{THRFin} and Corollary \ref{CPAb} (1) and (3), we can compute the abelianizations of the groups in consideration in Item (1). Indeed,

If $\mathfrak{S}(Q)=\mathcal S_q$, then $(\HA)_{\ab} = \Z^s \times ({\mathcal S_q})_{\ab}=
\frac{\mathbb Z}{2 \mathbb Z}  \times \Z^s $.

If $\mathfrak{S}(Q)=\mathcal A_q$, then $(\HA)_{\ab} = \Z^s $.

This implies that these groups are not isomorphic.

\medskip

$\pbul$ Moreover, one can construct subgroups of $\IET$ with the properties of Item (1) and only $2$ generators. Indeed,

Let $\mathfrak S (\QR)=\langle (1,2) \rangle< \mathcal S_q$ ($q\geq 5$) and  $A=\langle \alpha \rangle$ ($\alpha \in \R/{}_{\dis\Z}\setminus \Q/{}_{\dis\Z}$).

\smallskip

We have $W= \langle \sigma, \tau \rangle= \mathcal S_q$ and $[W,W]= [\mathcal S_q, \mathcal S_q]=\mathcal A_q$. Then, Theorem \ref{THF0} (1) applies and  $\HA$ has the required properties. In addition, by Corollary \ref{CPAb} (2), the abelianization of $\HA$ is $ \frac{\mathbb Z}{2 \mathbb Z} \times \Z $.
\end{rema}

\vfill \eject
\section{Appendix} \label{App} 

\subsection{Derived subgroup and Abelianization} 

\subsubsection{Definitions-Notations} Let $\Gamma$ be a group and $a,b,h\in \Gamma$. 

The \textbf{commutator} of $a$ and $b$ is the element of $\Gamma$
$$\boldsymbol{[a,b]:= ab a^{-1}b^{-1}}$$
One can check that $[b,a]=[a,b]^{-1}$, $h[a,b]h^{-1}=[hah^{-1},hbh^{-1}]$ and  $ba=[b,a]ab=ab[b^{-1}, a^{-1}].$

\smallskip

\noindent If $G$ and $K$ are subgroups of $\Gamma$, we set 
$\dis \boldsymbol{[G,K] :=\langle [g,k] \text{ with }  g \in G \text{ and } k \in  K \rangle}=[K,G]$ since  $[g,k]^{-1} =[k,g]$.

\smallskip

The \textbf{derived subgroup} of $\Gamma$ is $\boldsymbol{\dis D(\Gamma):=[\Gamma,\Gamma]}$.

\smallskip

Inductively, we define the \textbf{$k^{\text{th}}$ derived subgroup} of $\Gamma$, noted $\boldsymbol{D^k(\Gamma)}$ by $$D^{k+1}(\Gamma) = D (D^k(\Gamma)) \text{ and } D^0(\Gamma)=\Gamma$$

The \textbf{abelianization} of $\Gamma$ is $\boldsymbol{\dis \Gamma_{\ab}:=\frac{\Gamma}{D(\Gamma)}}$ \ and \ $\dis \boldsymbol{P_{\Gamma_{\ab}}}: \Gamma \to \Gamma_{\ab}$ denotes the canonical projection  $\gamma \mapsto \gamma. D(\Gamma)$.

\subsubsection{Basic properties} \label{BaseDG} \ 

\begin{enumerate}
\item If $K<H$ then $D^n(K)< D^n(H)$, $\forall n\in \N$.
\item Let $\phi : G \to G'$ be a group homomorphism then $\phi(D^n(G))=D^n(\phi(G))$, $\forall n\in \N$.
\item $D^n(H)$ is stable by any automorphism of $H$, hence $D^n(H)$ is normal.
\end{enumerate}

\subsubsection{Generators of the derived group}\label{geneDGFolk} 

Let $G$ be a group and $S$ be a generating set for $G$.

\smallskip

\noindent \textsf{We claim that any element of $[G,G]$ is the product of a finite number of conjugates of commutators of elements of $S$.}

\begin{proof} Indeed, let $g, k\in G$. \\ We first prove that $[g, k]$ is a product of conjugates of elements of the form $[a,k]$ with $a\in S$.

\begin{quote} 

\ \ \ By definition of $S$, $g$ can be written as $g=s_1 \cdots s_{p_g}$ with $s_i\in S$ and $p_g\in \N$. We argue by induction on $p$, supposing that for any  $f\in G$ that is a product of at most $p-1$ elements of $S$, the commutator $[f,k]$ is a product of conjugates of $[a,k]$ with $a\in S$.

A straightforward calculus leads to  $[g_1 \cdots g_p, k] = g_1 [g_2 \cdots g_p, k] g_1^{-1} [g_1,k]$ and by induction hypothesis, it holds that $[g_2 \cdots g_p, k] =\prod h_i [a_i,k]h_i^{-1}$.
\end{quote}

It remains to prove that any $[b,g]$ with $b\in S$, decomposes in conjugates of commutators of elements in $S$. This follows from the fact that $[b,g]= [g,b]^{-1}$ and the previous argument shows that $[g,b]$ has the required decomposition.
\end{proof}

\subsection{Semi direct and wreath products} 

\smallskip

In this subsection, we consider a \textbf{semi direct product} $H=N\rtimes G$ as a group $H$ equipped with $N\lhd H$ and $G<H$ such that $H=NG$ and $N\cap G =\{\id\}$.

\smallskip

In this internal context, the \textbf{wreath product} $L\wr G $ is the group $H=\bigoplus\limits_{g\in G} L^g \rtimes G$ where $L$ and $G$ are subgroups of $H$ such that the subgroups $L^g:= gL g^{-1}$, $g\in G$ are pairwise almost disjoints and commuting subgroups (i.e if $g'\not=g$ then $L^g\cap L^{g'}=\{\id\}$ and any element of $L^g$ commutes with any element of $L^{g'}$).

\subsubsection{Derived subgroup of a semi direct product}It is well known that 
$$\dis\boldsymbol{D(N\rtimes G) = \langle [N,N],[N,G] \rangle \rtimes [G,G] }$$

\subsubsection{Abelianization of semi direct and wreath products}  \ 

\begin{prop} \ 

\begin{enumerate}
\item The group $P_{N_{\ab}}\bigl([N,G]\bigr)$ is a normal subgroup of $N_{\ab}$ and $$\dis (N\rtimes G)_{\ab} = \frac{N_{\ab}}{P_{N_{\ab}}\bigl( [N,G] \bigr)} \times G_{\ab}$$
\item $\dis (L\wr G)_{\ab} = L_{\ab} \times G_{\ab}$.
\end{enumerate}
\end{prop}

\begin{proof} \ 
\begin{enumerate}
\item $P_{N_{\ab}}\bigl([N,G]\bigr)$ is a normal subgroup of $N_{\ab}$ since $[N,G]$ is a normal subgroup of $N$.

\medskip

As $ng \in D(H) \Lra g\in D(G) \text{ and } n \in \langle [N,N], [N,G]\rangle$, it holds that 
$$H_{\ab} =  \frac{N}{\langle [N,N], [N,G] \rangle} \times G_{\ab} $$
In addition, $[N,N]<\langle [N,N], [N,G] \rangle$ are normal subgroups of $N$, so by the third isomorphism theorem, we have 
$$ \frac{N}{\langle [N,N],[N,G] \rangle} = \frac{N_{\ab} }{V} \text{ where }$$
$$\dis V=\frac{\langle [N,N],[N,G] \rangle} {[N,N]} = P_{N_{\ab}}\bigl( \langle [N,N],[N,G] \rangle \bigr)= P_{N_{\ab}}\bigl([N,G]\bigr)$$

\item 

$(L\wr G)_{\ab} =  (\overbrace{\bigoplus\limits_{g\in G} L^g }^{N} \ \rtimes G)_{\ab} =G_{\ab} \times W$ where, according to Item (1), \vskip -4mm
$$W= \frac{N_{\ab}}{P_{N_{\ab}}\bigl([N,G]\bigr)}
= \frac{\bigoplus\limits_{g\in G} g(L_{\ab})g^{-1}}{P_{N_{\ab}}\bigl([N,G] \bigr)}=L_{\ab}$$ 
since \ 
$gL_\ab g^{-1} =\{g \ \ell.[N,N] \ g^{-1}, \ \ell \in L\}=\{g\ell g^{-1} \ g[N,N]g^{-1} , \ \ell \in L\}$

$=\{\ell \ [\ell ^{-1}, g][N,N], \ \ell \in L\}
\equiv L_\ab  \mod P_{N_{\ab}}\bigl([N,G]\bigr)$.
\end{enumerate}
\vskip -3mm \end{proof}

\subsection{Solvable groups} 
\begin{defi}
Let $n\in \N$, a group $G$ is \textbf{$n$-solvable} if $D^n(G)=\{\id\}$ and $D^{n-1}(G)\not=\{\id\}$  ($n\geq 1$), and $n$ is denoted by $\cl(G)$.
\end{defi}

\begin{prop} \label{derlen}\ 
\begin{enumerate}

\item If $H$ is $n$-solvable then any subgroup and any quotient of $H$ is $(\leq n)$-solvable .

\smallskip

\item Let $H$ be a group and $N$ be a normal subgroup of $H$. 

$H$ is solvable if and only if $N$ is solvable and $H/N$ is solvable.

In addition $\max(\cl (N), \cl (H/N)) \leq \cl (H) \leq \cl (N)+\cl (H/N)$.

\smallskip

\item $L\wr G$ is solvable if and only if $L$ and $G$ are solvable.

In addition $\max(\cl (L),\cl (G)) \leq \cl (H) \leq \cl (L)+\cl (G)$.
\end{enumerate}
\end{prop}

\begin{proof} \ 
\begin{enumerate}
 
\item Let $K<H$, by the basic properties \ref{BaseDG} we have  $D^n(K)<D^n(H)=\{\id\}$.

Let $H/N$ be a quotient of $H$ and $p : H\twoheadrightarrow H/N$. By Property \ref{BaseDG} (2), we have $p(D^n(H))= D^n(p(H))=D^n(H/N)$.

\medskip

\item Suppose that $N$ is $n$-solvable and $H/N$ is $m$-solvable then by Property \ref{BaseDG} (2), it holds that $p(D^m(H)) =D^m(p(H))= D^m(H/N)=\{\id\}$. This means that $D^m(H)<N$, therefore $D^{m+n}(H)<D^n(N)=\{\id\}$.

\medskip

\item According to the previous item, $L\wr G$ is solvable if and only if 
$\bigoplus\limits_{g\in G} L^g$ and $G$ are solvable. In addition, $\bigoplus\limits_{g\in G} L^g$ is $n$-solvable if and only if  $L$ is $n$-solvable. Therefore, $L\wr G$ is solvable if and only if 
$L$ and $G$ are solvable and $\max(\cl (L), \cl (G)) \leq \cl (H) \leq \cl (L)+\cl (G)$.

\end{enumerate}
\vskip -0.5cm \end{proof}

\subsection{Virtually solvable groups}  \label{AppVR} \

\smallskip

This subsection resume the solution given by Y. Cornulier for the  Mathematics Stack Exchange conversation \textbf{"If $N$ and $G/N$ are virtually solvable then $G$ is virtually solvable".}

\medskip

\begin{defi} Let $G$ be a group.
\begin{itemize}
\item The group $G$ is \textbf{virtually solvable} if contains a finite index subgroup that is solvable.
\item A finite index normal solvable subgroup of $G$ is \textbf{maximal} if its index is minimal.
\end{itemize}
\end{defi}

\smallskip

As every finite index subgroup contains a normal subgroup also of finite index, being virtually solvable is equivalent to containing a finite index normal subgroup that is solvable.

\medskip

\begin{prop} 
\begin{enumerate}
\item Any subgroup of a virtually solvable group is virtually solvable.
\smallskip
\item Let $G$ be a group and $N$ be a normal  subgroup of $G$.
If  $N$ and $G/N$ are virtually solvable then $G$ is virtually solvable.
\end{enumerate}
\end{prop}

\begin{proof} \ 

\begin{enumerate}
\item Let $G$ be a group and $K<G$ be a solvable subgroup of finite index. Let $\Gamma <G$.

\smallskip

The group $\Gamma \cap K$ has finite index in $\Gamma$, since the map $\dis \left\{ \begin{array}{ll}
\Gamma\diagup { \hskip -1mm}_{\dis \Gamma \cap K} \to  G\diagup {\hskip -1mm }_{\dis K} \cr \gamma. \  \Gamma \cap K \ \mapsto  \gamma.\ K \end{array}\right.$ is injective, and
$G\diagup {\hskip -1mm }_{\dis K}$ is finite. 
In addition,  $\Gamma \cap K$ is solvable as a subgroup of the solvable group $K$.

\medskip

\item  $\boldsymbol{(a)}$  We consider the case where $N$ is solvable and $G/N$ is virtually solvable. Let $\mathcal R$ be a finite index solvable subgroup of $G/N$, noting that $R:=p^{-1}(\mathcal R)$ is a subgroup of $G$ that contains $N$, 
the associated exact sequence 
$$1\to \overbrace{\ \ N \ \ }^{\text{solvable}} \hookrightarrow \quad  G \quad \ra\limits^{p} \quad \overbrace{\ \ G/N\  \ }^{\text{\tiny virtually solvable}} \to 1$$ leads to 
$$1\to \  \underbrace{\ \ N \ \ }_{\text{solvable}}  \hookrightarrow \ \ {R=p^{-1}(\mathcal R) }\ra  \underbrace{\mathcal R= R/N}_{\text{solvable}} \to 1  $$

By the classical argument for solvable groups, $R$ is solvable and 
we claim that $R$ has finite index in $G$. Indeed, the third isomorphism theorem implies that
$$\frac{G/N}{\mathcal R} = \frac{G/N}{R/N}\simeq G/R$$.

$\boldsymbol{(b)}$ We consider the case where $G\diagup{\hskip -1mm}_{\dis N}$ is solvable and $N$ is finite.

\medskip

Let $C(N):=\{h\in G \ | \ hn=nh, \forall n\in N \}$ be the centralizer of $N$ in $G$.

\smallskip

\ \ $(i)$ \textsf{We first prove that $C(N)$ has finite index in $G$.}

As $C(N) = \bigcap_{n\in N} C(n)$ and $N$ is finite, it suffices to prove that $C(n)$ is of finite index in $G$, for every $n\in N$. Indeed, let $n\in N$.  
For any $f,g \in G$, we have $$\dis [n,f]=[n,g] \Lra n f n^{-1} f^{-1}=  n g n^{-1} g^{-1} \Lra 
g^{-1} f n^{-1} =   n^{-1} g^{-1} f  \Lra g^{-1} f \in C(n)$$
In addition, $\dis \mathcal F:=\left\{ [n,f] , \  f\in G \right\}$ is finite as a subset of $N$, since $\dis [n,f] = n\overbrace{f n^{-1} f^{-1}}^{\sct \in N} \in N$. Let's write $\mathcal F=\left\{ [n,f_i] , 1 \leq i \leq p\right\}$.

\smallskip 

Let $g\in G$, there exists $ 1 \leq i \leq p$ such that $ [n,g]=[n,f_i]$, hence such that $f_i ^{-1} g \in C(n)$. This implies that $C(n)$ has finite index in $G$ and so does $C(N)$.

\medskip

$(ii)$ \textsf{We prove that $C(N)$ is solvable.} 

Let $p : G \to G/N$ be the canonical projection. The group $C(N)/N\cap C(N) \simeq p(C(N))$ is solvable as a subgroup of the solvable group $G/N$. And, we have the exact sequence $$ 1 \to \underbrace{N\cap C(N)}_ {\text{abelian}} \to C(N)\to \underbrace{C(N)/N\cap C(N)}_{\text{solvable}} \to 1$$  By the classical argument for solvable groups, $C(N)$ is solvable. 

\medskip

In conclusion, $(i)$ and $(ii)$ imply that $G$ is virtually solvable.

\bigskip
$\boldsymbol{(c)}$ We consider the case where $N$ is virtually solvable and  $G\diagup{\hskip -1mm}_{\dis N}$ is solvable.

Let $M$ be a solvable normal subgroup of $N$ of minimal finite index in $N$ among the normal solvable subgroups of finite index in $N$.

We claim that $M$ is normal in $G$. Otherwise, there exists $g\in G$ such that $M^g=gMg^{-1}\not= M$. Then $MM^g$ is a solvable normal subgroup of $N$, and its index in $N$ is strictly less than the index of $M$. 

Therefore, we have
$$1\to \underbrace{N/M}_{\text{finite}} \to G/M \to \frac{G/M}{N/M} \simeq \underbrace{G/N}_{\text{solvable}}\to 1$$
Hence, by Item $(b)$, $G/M$ is virtually solvable. 

Finally, $M$ is solvable and $G/M$ is virtually solvable, then by Item $(a)$, $G$ is virtually solvable.

\medskip

$\boldsymbol{(d)}$ In the general case, we have:

$$\dis 1\to \overbrace{\ \ N \ \ }^{\text{\tiny virtually solvable}} \hookrightarrow \ \  G \ \  \ra\limits^{p} \ \ \overbrace{\ \ G/N\ \ }^{\text{\tiny virtually solvable}}\to 1 $$
leading, by choosing  $\mathcal R$  a finite index solvable subgroup of $G/N$, to 
$$1\to  \ \  \underbrace{\ \ N \ \ }_{\text{\tiny virtually solvable}}  \hookrightarrow \ \ R=p^{-1}(\mathcal R)  \ra  \underbrace{\mathcal R= R/N}_{\text{solvable}} \ \to 1  $$

We conclude by Item (c) that $R$ is virtually solvable and it as finite index (see (a)) in $G$. Hence, the group $G$ is also virtually solvable.

\end{enumerate}
\end{proof}

\bigskip

\bibliographystyle{alpha}
\bibliography{RefIET}
\end{document}